\documentclass[10pt,a4paper,reqno,twoside,intlimits]{amsart}
\usepackage{hyperref}
\usepackage{fullpage}
\usepackage{amsfonts}
\usepackage{amssymb}
\usepackage{amsthm}
\usepackage{latexsym} 
\usepackage[latin1]{inputenc}
\allowdisplaybreaks[2]
\usepackage{mathtools}
\newcommand*{\transp}[2][-3mu]{\ensuremath{\mskip1mu\prescript{\smash{\tau \mkern#1}}{}{\mathstrut#2}}}%

\def\N{\mathbb N}
\def\R{\mathbb R}
\def\C{\mathbb C}

\def\D{\mathcal D}
\def\CC{\mathcal C}\def\O{\mathcal O}

\def\E{\mathcal E}

\def\G{\mathcal G}

\def\U{\mathcal U}
\def\O{\mathcal O}
\def\M{\mathcal M}
\def\W{\mathcal W}
\def\eps{\varepsilon}
\def\NN{\mathcal N}

\newcommand{\crb}{\mathcal{V}}

\DeclareMathOperator{\id}{id}
\DeclareMathOperator{\Id}{Id}
\DeclareMathOperator{\WF}{WF}
\DeclareMathOperator{\real}{Re}
\DeclareMathOperator{\imag}{Im}
\DeclareMathOperator{\supp}{supp}

\DeclareMathOperator{\Char}{Char}
\DeclareMathOperator{\spanc}{span}

\theoremstyle{plain}
\newtheorem{Thm}{Theorem}[section]
\newtheorem{Prop}[Thm]{Proposition}
\newtheorem{Lem}[Thm]{Lemma}
\newtheorem{Cor}[Thm]{Corollary}

\theoremstyle{definition}
\newtheorem{Def}[Thm]{Definition}
\newtheorem{Ex}[Thm]{Example}
\newtheorem{Rem}[Thm]{Remark}

\numberwithin{equation}{section}
\title{Ultradifferentiable CR Manifolds}
\author{Stefan F\"urd\"os}
\address{Faculty of Mathematics, University of Vienna, Oskar-Morgenstern-Platz 1, 1090 Vienna, Austria}
\curraddr{Department of Mathematics and Statistics, Masaryk University,  Kotlarska 2, 611 37 Brno, Czech Republic}
\email{stefan.fuerdoes@univie.ac.at}

\begin{document}
\begin{abstract}
	In this article the notion of ultradifferentiable CR manifold is introduced 
	and an ultradifferentiable regularity result for finitely nondegenerate %%reflection principle
	CR mappings is proven.
	Here ultradifferentiable means with respect to Denjoy-Carleman classes
	defined by weight sequences.
	Furthermore the regularity of infinitesimal CR automorphisms on ultradifferentiable
	abstract CR manifolds is investigated.
\end{abstract}
\maketitle
\section{Introduction}\label{Intro}
The primary focus of this article is the study of %certain results on 
the regularity of CR mappings.
%which have been traditionally referred to as reflection principles.
%The epynom of statements of this kind is the classical Schwarz reflection principle, which in fact may be viewed as a regularity result: 
%Any real valued continuous function on the real line that extends holomorphically to one side is actually real analytic.
%Note that $\R\subseteq\C$ is a totally real submanifold and hence all continuous real valued function can be considered as CR mappings on $\R$.
%
%The Schwarz reflection principle can easily be generalized to mappings between totally real submanifolds of $\C^n$.
%However it was a surprise when  in the second half of the last century an increasing number of reflection principles for CR mappings between more general CR submanifolds were proven,
%beginning with the epochal theorem of Fefferman \cite{MR0350069} on the smooth extension of biholomorphisms of bounded strictly pseudoconvex domains in $\C^n$.
%Among the important results on the boundary regularity that were shown after the theorem of Fefferman we would like to mention the reflection principle of Nirenberg-Webster-Yang \cite{MR562738} 
%and  the reflection principle for CR diffeomorphisms on essential finite real analytic hypersurfaces of Baouendi-Jacobowitz-Treves \cite{MR808223} to name only a few.
%
Looking at the literature concerning this problems, one observes that most theorems about the regularity of CR mappings are of a similar form which can be summarized as follows:
We consider a CR mapping $H$ between two CR submanifolds $M$ and $M^\prime$ 
with some a-priori regularity 
that extends to a holomorphic mapping defined on  a wedge with edge $M$. 
If the mapping and/or the manifolds satisfy certain nondegeneracy conditions at some point then it is proven that $H$ is actually of optimal
regularity near this point, that is smooth if $M$ and $M^\prime$ are smooth, or real-analytic if the manifolds are real-analytic.
We should mention that the nondegeneracy assumptions are heavily tailored towards the methods applied in the various different proofs.
In particular, it is worth noting that in most instances the conditions in the smooth setting differ sharply from
those used in the analytic category.
In the case of smooth CR manifolds the fundamental contributions are the pioneering works of Fefferman \cite{MR0350069} and Nirenberg-Webster-Yang \cite{MR562738}.
We should mention that in the analytic setting surprisingly 
 weak assumptations often suffice, c.f.\ e.g.\ the classical results of Baouendi-Jacobowitz-Treves \cite{MR562738}, 
Huang \cite{MR1421212} and 
Pin{\v c}uk \cite{MR496595}.

One of the rare cases, where under the identical assumptions 
it has been possible to show that $H$ is smooth if the manifolds are smooth 
and analytic if $M$ and $M^\prime$ are both analytic manifolds, 
have been the results of Bernhard Lamel \cite{MR1861300,MR2085046}. 
He proved that every finitely nondegenerate CR mapping between two generic submanifolds that extends holomorphically  is smooth and even analytic if both manifolds are real-analytic.

Recently Berhanu-Xiao \cite{MR3405870} were able to strengthen this result in the smooth case by relaxing
partially its assumptions.
They require only the target manifold to be an embedded CR manifold, the source manifold could be only an abstract CR manifold.
The finitely nondegenerate condition on the mapping remains unchanged but the holomorphic extension obviously makes no sense in this situation. 
It is replaced in the theorem of Berhanu-Xiao with the assumptation
that the fibers of the wavefront set of $H$ do not include opposite directions.

This microlocal assumption is automatically satisfied in the embedded setting if %holomorphic 
extension to a wedge is assumed since Baouendi-Chang-Treves \cite{MR723811} showed that for CR distributions
on CR submanifolds of $\C^N$ the holomorphic extension into wedges is in fact a microlocal condition, which they used to define the hypoanalytic wavefront set of CR distributions. 
It coincides with the analytic wavefront set if the manifold is analytic. If the manifold is only smooth then the hypoanalytic wavefront set includes the smooth wavefront set.

%Moreover, Lamel used in the proof of the smooth version \cite{MR2085046} of his reflection principle basic techniques from microlocal analysis in all but name.

Since the results of Lamel and Berhanu-Xiao suggest that finite nondegeneracy preserves regularity quite well,
the following question arises naturally.
Given a subsheaf $\mathcal{A}$ of the sheaf of smooth functions we may ask that if in the formulation of the theorem of Lamel 
the manifolds are assumed to be of class $\mathcal{A}$, does it follow that the CR mapping has to be of 
class $\mathcal{A}$ as well?

Of course we have to assume that $\mathcal{A}$ satisfies certain properties.
First of all, in order for the conjecture above to make sense, $\mathcal{A}$ must be closed under composition and 
the implicit function theorem must hold in the category of mappings of class $\mathcal{A}$.
Furthermore if we try to modify the existing proofs in the smooth category then we need some version of
$\mathcal{A}$-wavefront set or more precisely a definition of $\mathcal{A}$-microlocal regularity.
We should note at this point that in both Lamel's proof and that of Berhanu-Xiao the characterization of the smooth wavefront set 
by almost-analytic extensions was heavily used as both relied on an almost-analytic version of the implicit function theorem.

We are mainly interested in  subsheafs of smooth functions that contains strictly the sheaf of real-analytic functions. 
We shall call the elements of such sheafs ultradifferentiable functions. 
Generally ultradifferentiable functions are determined either by estimates on its derivatives or its Fourier transform.
The most well-known examples of ultradifferentiable classes are the Gevrey classes, see e.g.\ \cite{MR1249275}.

Here we consider  the category of socalled Denjoy-Carleman classes, which are defined
in the following way.
If $\M=(m_j)_j$ is a sequence of positive real numbers then the Denjoy-Carleman class associated with $\M$ consists of those 
smooth functions that satisfy the following generalized Cauchy estimate
\begin{equation}\label{initialEst}
\bigl\lvert \partial^\alpha f(x)\bigr\rvert\leq C h^{\lvert\alpha\rvert}m_{\lvert\alpha \rvert}\lvert\alpha\rvert !
\end{equation}
on compact sets, where $C$ and $h$ are constants independent of $\alpha$. 
We will also say that a smooth function $f$ obeying \eqref{initialEst}
is of class $\{\M\}$.
In particular, if $\M=(j!^s)_j$ then the associated Denjoy-Carleman class to $\M$
is the Gevrey class of order $s+1$.

Examining the literature concerning the Denjoy-Carleman classes and their properties one can observe that
 stability conditions of the associated class correlate with properties of the weight sequence.
For example, we know that, if $\M$ is a regular weight sequence in the sense of
 \cite{zbMATH03751341}, then  the Denjoy-Carleman class associated to $\M$ 
 is closed under composition, solving ordinary differential equations 
 and the implicit function theorem holds in the class, c.f.\ e.g.\ \cite{MR3462072}. 
Hence for regular sequences $\M$ we can consider manifolds of Denjoy-Carleman type.
We shall say such a manifold is an ultradifferentiable manifold of class $\{\M\}$.

On the other hand, H\"ormander \cite{MR0320486} introduced the ultradifferentiable
wavefront set for distributions defined on open subsets of the euclidean space.
But %his definition is a little bit too general for the purposes of this article.
since he worked under comparatively weak conditions on the weight sequence H\"ormander was only able to define the 
ultradifferentiable wavefront set $\WF_\M u$ of distributions $u$ on real-analytic manifolds but not distributions defined on 
ultradifferentiable manifolds.  

However using Dyn'kins  characterization of ultradifferentiable functions by
almost analytic extensions \cite{zbMATH03751341,MR0587795}
we were able in \cite{Fuerdoes1} to develop a geometric theory for the ultradifferentiable wavefront set.
In particular, if the weight sequence is regular, the ultradifferentiable wavefront 
set of a distribution on an ultradifferentiable manifold is shown to be well defined.

With these results at hand and an $\M$-almost analytic version of the almost-analytic implicit function theorem used in
Lamel \cite{MR2085046} and Berhanu-Xiao \cite{MR3405870}
it is possible to prove the ultradifferentiable version of the regularity result of Lamel:
\begin{Thm}\label{Reflectionsprinciple1}
	Let $M\subseteq\C^N$ and $M^\prime\subseteq\C^{N^\prime}$ be two generic ultradifferentiable 
	submanifolds of class $\{\M\}$, 
	$p_0\in M$, $p^\prime_0\in M^\prime$ and $H\!:(M,p_0)\rightarrow (M^\prime,p_0^\prime)$
	a $\CC^{k_0}$-CR mapping that is $k_0$-nondegenerate at $p_0$. Suppose furthermore that $H$
	extends continuously to a holomorphic map in a wedge $\W$ with edge $M$.
	Then $H$ is ultradifferentiable of class $\{\M\}$ in a neighbourhood of $p_0$.
\end{Thm}
For the definition of finite nondegeneracy of a CR mapping 
 we refer to the beginning of  section \ref{sec:ultraRefl}.

More precisely this paper is structured as follows.
In section \ref{regDC} the results on regular Denjoy-Carleman classes and ultradifferentiable manifolds that are needed are discussed.
In section \ref{UltraWF} we first recall the results from Dyn'kin \cite{zbMATH03751341,MR0587795}
on the almost analytic extension of ultradifferentiable functions.
Furthermore we give the definition and basic results on the ultradifferentiable
wavefront set according to H\"ormander \cite{MR1996773} and close the section
by presenting the geometric theory for the ultradifferentiable wavefront set 
given in \cite{Fuerdoes1}.

In section \ref{sec:ultraCR} basic definitions and first results on ultradifferentiable CR manifolds are given, 
whereas the proofs of Theorem \ref{Reflectionsprinciple1} and of ultradifferentiable versions of other regularity results of Lamel and Berhanu-Xiao are presented in
section \ref{sec:ultraRefl}. 
The last section is devoted to present essentially the generalization
of \cite{MR3593674} concerning the smoothness of infinitesimal CR automorphisms 
to regular Denjoy-Carleman classes.
 We end by examining smooth infinitesimal CR automorphisms 
 on formally holomorphic nondegenerate quasianalytic CR submanifolds.
 
 The author was supported by the Austrian Science Fund FWF, international cooperation  project Nr.\ I01776 and the Czech Science Foundation GACR grant 17-19437S. 
\section{Regular Denjoy-Carleman classes}\label{regDC}
%In the following sections let $\Omega\subseteq\R^n$ be an open set.
In this section we summarize the results for Denjoy-Carleman classes that we need in the following. 
For a more detailed presentation see \cite{Fuerdoes1}.
Note that, unless stated otherwise, $\Omega\subseteq\R^n$ will be an open set.
\begin{Def}\label{Regular}
A sequence $\M=(m_k)_k$ is a regular weight sequence iff it satisfies the following conditions.
\begin{align*}
\tag{\text{M}1} &m_0=m_1=1\label{normalization}\\
\tag{\text{M}2} \label{derivclosed} &\sup_k\sqrt[\leftroot{1}k]{\frac{m_{k+1}}{m_k}}<\infty\\
\tag{\text{M}3}\label{stlogconvex} &m_k^2\leq m_{k-1}m_{k+1}\qquad k\in\N\\
\tag{\text{M}4}\label{analyticincl} &\lim_{k\rightarrow\infty}\sqrt[\leftroot{2}k]{m_k}=\infty
\end{align*}
\end{Def}
\begin{Def}
Let $\M$ be a regular weight sequence.
Then we say that a smooth function $f\in\E(\Omega)$ is ultradifferentiable of 
class $\{\M\}$ iff for all compact sets $K\subseteq\Omega$ there are constants 
$C,h>0$ such that
\begin{equation}\label{DCdefiningEst}
\bigl\lvert \partial^\alpha f(x)\bigr\rvert\leq Ch^{\lvert\alpha\rvert}m_{\lvert\alpha\rvert}
\lvert\alpha\rvert!
\end{equation}
for all $x\in K$.
The space of all ultradifferentiable functions of class $\{\M\}$ is denoted by 
$\E_\M(\Omega)$. It is sometimes also called the Denjoy-Carleman class associated to
$\M$.
\end{Def}
\begin{Ex}
If $s> 0$ consider the regular weight sequence $\M^s=(k!^s)_k$. 
Its associated Denjoy-Carleman class is the Gevrey class
$\G^{s+1}(\Omega)=\E_{\M^s}(\Omega)$ of order $s+1$ on $\Omega$,
 c.f.\ \cite{MR1249275}.

On the other hand, the constant sequence $\M^0=(1)_k$ gives the space $\O(\Omega)$ of 
real-analytic functions on $\Omega$. Note that $\M^0$ is not regular in the
sense of Definition \ref{Regular}.
\end{Ex}
We consider here only regular weight sequences but might occasionally omit the word ``regular''.

If $\M$ and $\NN=(n_k)_k$ are two weight sequences then we write $\M\preccurlyeq \NN$ iff
there is a constant $Q$ such that $m_k\leq Q^kn_k$.
It holds that $\E_{\M}\subseteq\E_{\NN}$ if and only if $\M\preccurlyeq\NN$.
Thus we see that \eqref{analyticincl} means that $\O\subsetneq\E_{\M}$ and \eqref{derivclosed} implies that $\E_{\M}$ is 
\emph{closed under derivation}, i.e.\ if $f\in\E_{\M}(\Omega)$ then
$\partial^\alpha f\in\E_{\M}(\Omega)$ for all multi-indices $\alpha\in\N_0^n$.
Furthermore we have
\begin{Lem}[c.f.\ Remark 2.5 in \cite{Fuerdoes1}]\label{HadamardLemma}
Let the Denjoy-Carleman class $\E_{\M}$ be closed under derivation closed and
suppose that $f\in\E_\M(\Omega)$ and $f(x_1,\dotsc ,x_{j-1},a,x_{j+1},\dotsc ,x_n)=0$ 
for some fixed $a\in\R$ and all $x_k$, $k\neq j$, with the property that 
$(x_1,\dotsc ,x_{j-1},a,x_{j+1},\dotsc ,x_n)\in\Omega$. 
Then there exists some $g\in\E_{\M}(\Omega)$ such that
	\begin{equation*}
	f(x)=(x_j -a) g(x).
	\end{equation*}
\end{Lem}

In fact, if $\M$ is a regular weight sequence then the associated Denjoy-Carleman class satisfies also
the following stability properties.
%Furthermore we have the following theorem.
\begin{Thm}\label{CMStability}
	Let $\M$ be a regular weight sequence and $\Omega_1\subseteq\R^m$ and 
	$\Omega_2\subseteq\R^n$ open sets. Then the following holds:
	\begin{enumerate}
		\item The algebra $\E_{\M}(\Omega)$ is \emph{inverse closed}, i.e.\
		 if $f\in\E_\M(\Omega)$ does not vanish at any point of $\Omega$ then
	$1/f\in\E_\M(\Omega)$
		(c.f.\ \cite{MR3462072} and the remarks therein).
		\item The class $\E_\M $ is \emph{closed under composition} (\cite{MR0158261} see also \cite{MR2061220}) i.e.\ let $F\!:\Omega_1\rightarrow\Omega_2$ be
		an $\E_\M$-mapping, that is each
		component $F_j$ of $F$ is ultradifferentiable of class $\{\M\}$ in $\Omega_1$, and $g\in\E_\M(\Omega_2)$.
		Then also $g\circ F\in \E_\M(\Omega_1)$.
		\item The \emph{inverse function theorem} holds in the Denjoy-Carleman class $\E_\M$ (\cite{MR531445}): 
		Let $F:\Omega_1\rightarrow \Omega_2$ be an $\E_\M$-mapping and $p_0\in\Omega_1$ such that 
		the Jacobian $F^\prime(p_0)$ is invertible. Then there exist neighbourhoods $U$ of $p_0$ in $\Omega_1$
		and $V$ of $q_0=F(x_0)$ in $\Omega_2$ and a $\E_\M$-mapping $G:V\rightarrow U$ such that $G(q_0)=p_0$
		and $F\circ G=\id_V$.
		\item The \emph{implicit function theorem} is valid in $\E_\M$ (\cite{MR531445}): Let $F:\R^{n+d}\supseteq\Omega\rightarrow \R^d$
		be a $\E_\M$-mapping and $(x_0,y_0)\in \Omega$ such that $F(x_0,y_0)=0$ and
		$\tfrac{\partial F}{\partial y}(x_0,y_0)$
		is invertible. Then there exist open sets $U\!\subseteq\!\R^n$ and $V\!\subseteq\! \R^d$ with 
		$(x_0,y_0)\!\in\! U\!\times\! V\!\subseteq\!\Omega$ and an $\E_\M$-mapping $G:\,U\rightarrow V$ 
		such that $G(x_0)=y_0$
		and $F(x,G(x))=0$ for all $x\in V$.
	\end{enumerate}
\end{Thm}
In particular we note that
$\E_\M(\Omega)$ is \emph{closed under solving ODEs}. More precisely we have the following result.
\begin{Thm}[\cite{MR1128962}, see also \cite{MR575993}]\label{ClosednessODE}
	Let $\M$ be a regular weight sequence, $0\in I\subseteq\R$ an open interval, 
	$U\subseteq \R^n$, $V\subseteq\R^d$ open and $F\in\E_\M(I\times U\times V)$. 
	
	Then the initial value problem
	\begin{align*}
	x^\prime(t)&=F(t,x(t),\lambda) & t&\in I,\,\lambda \in V\\
	x(0)&=x_0 &x_0 &\in U
	\end{align*}
	has locally a unique solution $x$ that is ultradifferentiable near $0$.
	
	More precisely, there is an open set $\Omega\subseteq I\times U\times V$ that contains the point 
	$(0,x_0,\lambda)$ and an $\E_\M$-mapping $x=x(t, y,\lambda): \Omega\rightarrow U$ such that
	the function $t\mapsto x(t,y_0,\lambda_0)$ is the solution of the initial value problem
	\begin{align*}
	x^\prime (t)&=F(t,x(t),\lambda_0)\\
	x(0)&=y_0.
	\end{align*}
\end{Thm}

Using Theorem \ref{CMStability} we are able to define
\begin{Def}
	Let $M$ be a smooth manifold and $\M$ a  weight sequence. We say that $M$ is an ultradifferentiable 
	manifold of class $\{\M\}$ iff there is an atlas $\mathcal{A}$ of $M$ that consists of charts such that 
	\begin{equation*}
	\varphi^\prime\circ\varphi^{-1}\in\E_{\M}
	\end{equation*}
	for all $\varphi,\varphi^{\prime}\in\mathcal{A}$.
\end{Def}
If $M\subseteq \R^N$ is an ultradifferentiable submanifold of class $\{\M\}$ then the following characterization is proven exactly as
the analogous result in the smooth setting.
\begin{Prop}
	Let $M\subset\R^N$ be a smooth manifold  of dimension $n$ and $p\in\M$ and $\M$ be a weight sequence.
	The following statements are equivalent:
	\begin{enumerate}
		\item The manifold $M$ is ultradifferentiable of class $\{\M\}$ near $p$.
		\item There are an open neighbourhood $U\subseteq\R^N$ of $p$ and an $\E_\M$-mapping 
		$\rho:\,U\rightarrow \R^{N-n}$ such that $d\rho$ has rank $N-n$ on $W$ and 
		\begin{equation*}
		\rho^{-1}(0)=M\cap U.
		\end{equation*}
	\end{enumerate}
\end{Prop}

A mapping $F\!:\, M\rightarrow N$ between two manifolds of class $\{\M \}$ is 
ultradifferentiable of class $\{\M \}$ iff $\psi\circ F\circ \varphi^{-1}\in\E_{\M}$
for any charts $\varphi$ and $\psi$ of $M$ and $N$, respectively.
Thus it is possible to consider the category of ultradifferentiable manifolds with all
the usual constructions like vector fields, differential forms and so on.
\begin{Def}
	Let $M$ be an ultradifferentiable manifold of class $\{\M\}$. We say that
	a smooth vector bundle $\pi: E\rightarrow M$ is an ultradifferentiable vector bundle of class $\{M\}$ iff
	for any point $p\in M$ there is a neighbourhood $U$ of $p$ and 
	a local trivialization $\chi$ of class $\{\M\}$ on $U$.
\end{Def}
\begin{Rem}\label{flatness}
	Let $E$ be an ultradifferentiable vector bundle of class $\{\M\}$. 
	Then $E$ can also be considered as a smooth vector bundle or as a vector bundle of class $\NN$
	for any weight sequence $\NN\succcurlyeq\M$.
	We observe in particular that a local basis of $\E_\M (M,E)$ is also a local basis of $\E_{\NN}(M,E)$ and $\E(M,E)$, respectively.
\end{Rem}
We denote by $\mathfrak{X}_\M(M)=\E_\M(M, TM)$ the Lie algebra of ultradifferentiable vector fields on $M$.
Note that, if $\M$ is a regular weight sequence, an integral curve of an ultradifferentiable vector field of class $\{\M\}$ 
%that is a vector field with 
%$\E_\M$-coefficients, 
is an $\E_\M$-curve by Theorem \ref{ClosednessODE}.

The next result is an ultradifferentiable version of Sussmann's Theorem \cite{MR0321133}.
\begin{Thm}\label{LocalSussman}
	Let $p_0\in\Omega$ and a collection $\mathfrak{D}$ of ultradifferentiable vector fields of class $\{\M\}$.
	Then there exists an ultradifferentiable submanifold $W$ of $\Omega$ through $p_0$ such that all vector fields
	in $\mathfrak{D}$ are tangent to $W$ at all points of $W$ and such that the following holds:
	\begin{enumerate}
		\item The germ of $W$ at $p_0$ is unique, i.e.\  
		if $W^\prime$ is an ultradifferentiable submanifold of $\Omega$ containing $p_0$ and to which all
		vector fields of $\mathfrak{D}$ are tangent at every point of $W^\prime$ then there is a neighbourhood 
		$V\subseteq\Omega$ of $p_0$ such that $W\cap V\subseteq W^\prime\cap V$.
		\item For every open set $U\subseteq\Omega$ containing $p_0$ there exists $J\in\N$ 
		and open neighbourhoods $V_1\subseteq V_2\subseteq U$ of $p_0$ such that every point $p\in W\cap V_1$
		can be reached from $p_0$ by a polygonal path of $J$ integral curves of vector fields in $\mathfrak{D}$ contained in $W\cap V_2$.
	\end{enumerate}
\end{Thm}
The proof of Theorem \ref{LocalSussman} is essentially the same as in the smooth setting, 
c.f.\ e.g.\ \cite{MR1668103}, due to Theorem \ref{ClosednessODE}.

The (unique) germ of the manifold $W$ will be denoted as the \emph{local Sussmann orbit} of $p_0$ relative to 
$\mathfrak{D}$. The local Sussman orbit does not depend on $\Omega$.

One of the main differences between the space of smooth functions and the space of 
real analytic functions is that in the smooth case there exist nontrivial test functions $\varphi\in\D(\Omega)$ whereas $\D\cap\O=\{0\}$.
Since the existence of functions of nontrivial test functions 
is equivalent to the existence
of nonzero flat functions, it makes sense to give the following definition in the ultradifferentiable setting.
\begin{Def}
	Let $E\subseteq\E(\Omega)$ be a subalgebra. We say that $E$ is quasianalytic iff 
	for $f\in E$ the fact that $D^\alpha f(p)=0$ for some $p\in\Omega$ and all $\alpha\in\N_0^n$ implies
	that $f\equiv 0$ in the connected component of $\Omega$ that contains $p$.
\end{Def}
In the case of Denjoy-Carleman classes quasianalyticity is characterized by the following theorem.
\begin{Thm}[Denjoy\cite{zbMATH02601219}-Carleman\cite{zbMATH02599917,zbMATH02598188}]
	The space $\E_\M(\Omega)$ is quasianalytic
	%, 
	%i.e.\ not a quasianalytic subalgebra of $\E(\Omega)$, 
	if and only if
	\begin{equation}\label{quasiCond}
	\sum_{k=1}^\infty\frac{m_{k-1}}{km_k}=\infty.
	\end{equation}
\end{Thm}
We say that a weight sequence is quasianalytic if it satisfies \eqref{quasiCond}
and non-quasianalytic if not.
\begin{Ex}
	Let $\sigma>0$ be a parameter. We define a family $\mathcal{N}^\sigma$ of  weight sequences by
	\begin{equation*}
	n_k^\sigma=\bigl(\log(k+e )\bigr)^{\sigma k}.
	\end{equation*}
	The weight sequence $\mathcal{N}^\sigma$ is quasianalytic if and only if 
	$0<\sigma\leq 1$, see \cite{MR2384272}.
\end{Ex}
%We are going to close this section with a proof of a quasianalytic version of Nagano's theorem \cite{MR0199865}.

If $\M$ is a quasianalytic regular weight sequence then it is possible to show a quasianalytic version of Nagano's theorem \cite{MR0199865}, c.f.\ \cite{Fuerdoes1}.
As in the case of the ultradifferentiable version of Sussmann's theorem the proof is just a straightforward adaptation of the proof of the classical result, see e.g.\
\cite{MR1668103}.

%We denote also by $\mathfrak{X}_\M(\Omega)$ the Lie algebra of ultradifferentiable vector fields on 
%$\Omega$.
\begin{Thm}\label{NaganoThm}
Let  $U$ be an open neighbourhood of $p_0\in\R^n$ and $\M$ a quasianalytic regular weight sequence.
Furthermore let $\mathfrak{g}$ be a Lie subalgebra of $\mathfrak{X}_\M(U)$
that is also an $\E_\M$-module, i.e.\ if $X\in\mathfrak{g}$ and $f\in\E_\M(U)$ then $fX\in\mathfrak{g}$.

Then there exists an ultradifferentiable submanifold $W$ of class $\{\M\}$ in $U$, such that
\begin{equation}\label{Naganoeq}
T_pW=\mathfrak{g}(p)\qquad \forall p\in W.
\end{equation}
Moreover, the germ of $W$ at $p_0$ is uniquely defined by this property.
\end{Thm}
As in the analytic category, c.f.\ \cite{MR1668103}, we have the following result.
\begin{Cor}\label{NaganoSussman}
	Let $\M$ be quasianalytic and $\mathfrak{D}\subseteq\mathfrak{X}_\M(\Omega)$ a collection of ultradifferentiable vector fields.
	If $\mathfrak{g}=\mathfrak{g}_\mathfrak{D}$ is the Lie algebra generated by $\mathfrak{D}$ and $p_0\in\Omega$ then the local Sussman orbit of $p_0$, relative to $\mathfrak{D}$, coincides with the
	local Nagano leaf of $\mathfrak{g}$.
\end{Cor}
\begin{proof}
	Let $W_N$ be a representative of the local Nagano leaf of $\mathfrak{g}$ at $p_0$ and $W_S$ 
	a representative of the local Sussman orbit of $p_0$, relative to $\mathfrak{D}$.
	By Theorem \ref{LocalSussman} (1) there exists an open neighbourhood $V$ of $p_0$ such that
	$W_S\cap V\subseteq W_N\cap V$. On the other hand $\mathfrak{g}(p)=T_pW_N$ for all $p\in W_N$ and
	$\mathfrak{g}(p)\subseteq T_pW_S$ at every $p\in W_S$, hence $\mathfrak{g}(p)=T_p W_S$ for 
	$p\in W_S\cap V$. The uniqueness part of Theorem \ref{NaganoThm} gives the equality of the local Nagano leaf and the local Sussman orbit.
\end{proof}
We want to close this section by showing how  the results pertaining the division of smooth functions in 
\cite[section 4]{MR3593674} transfer to the category of ultradifferentiable functions of class $\{\M\}$.
This is possible because these classes are closed under division by a coordinate, 
i.e.\ Lemma \ref{HadamardLemma}.
\begin{Lem}
	Let $\lambda$ be  an ultradifferentiable function of class $\{\M\}$ defined near $0\in\R$ 
	that is non-flat at the origin, i.e.\
	there is a positive integer $k\in\N$ such that $\lambda^{(j)}(0)=0$ for all integers $0\leq j\leq k-1$ and 
	$\lambda^{(k)}(0)\neq 0$. 
	Further assume that there is a locally integrable function $u$ defined near $0$ such that
	the product $f=\lambda u$ is of class $\{\M\}$ in some neighbourhood of the origin.
	
	Then $u$ is ultradifferentiable of class $\{\M\}$ near the origin.
\end{Lem}
\begin{proof}
	First, we note that the zero of $\lambda$ at $0$ is isolated. Therefore we restrict ourselves to an open interval $I$ 
	that contains the origin and such that $0$ is the only zero of $\lambda$ on $I$. Iterating Lemma \ref{HadamardLemma} we see that there is a function $\tilde{\lambda}$ of class $\{\M\}$
	defined near $0$ such that 
	$\tilde{\lambda}(0)\neq 0$ and
	\begin{equation*}
	\lambda(x)=x^k\tilde{\lambda}(x).
	\end{equation*}
	In order to proceed we want a similar decomposition of $f$. But,  since we are not able to say anything apriori
	about the
	values of the derivatives of $f$ at the origin, we can only find an ultradifferentiable function $f_1$ such that
	\begin{equation*}
	f(x)=xf_1(x)
	\end{equation*}
	in a neighbourhood of $0$. If $k>1$ then we would have that
	\begin{equation*}
	u(x)=x^{1-k}\frac{f_1(x)}{\tilde{\lambda}(x)}
	\end{equation*}
	in a punctured neighbourhood of $0$. Hence, if $f_1(0)\neq 0$ then $u\sim x^{1-k}$ for $x\rightarrow 0$. 
	This is a contradiction to the assumption that $u$ is locally integrable. 
	Therefore $f_1(0)=0$ and there has to be a function $f_2$ of class $\{\M\}$ such that
	$f(x)=x^2f_2(x)$ near $0$. Repeating this argument if necessary, we obtain that there is a function $f_k$
	ultradifferentiable of class $\{\M\}$ defined near the origin such that
	\begin{equation*}
	f(x)=x^kf_k(x).
	\end{equation*}
	
	It follows that
	\begin{equation*}
	u(x)=\frac{f_k(x)}{\tilde{\lambda}(x)}
	\end{equation*}
	in some neighbourhood of $0$.
\end{proof}
\begin{Prop}\label{DivProp}
	Let $p_0\in\R^n$ and $\lambda$  an ultradifferentiable function of class 
	$\{\M\}$ defined in a neighbourhood of $p_0$ and $\lambda(p_0)=0$.
	Suppose that $\lambda^{-1}(0)$ 
	is a hypersurface of class $\{\M\}$ near $p_0$ and that  there are $v\in\R^n$ and 
	$k\in\N$ such that $\partial_v^j(p)= 0$ for $0\leq j<k$ and 
	$\partial^k_v(p)\neq 0$ for all $p\in\lambda^{-1}(0)\cap U$ where $U$ is a neighbourhood of $p_0$.
	
	If $u$ is a locally integrable function defined near the origin in $\R^n$ such 
	that $\lambda\cdot u=f$ is ultradifferentiable of class $\{\M\}$ near $p_0$ then
	$u$ has also to be of class $\{\M\}$ in some neighbourhood of $p_0$.
\end{Prop}
\begin{proof}
	We can choose ultradifferentiable coordinates $(x_1,\dotsc,x_{n-1},x_n)=(x^\prime,x_n)$
	in a neighbourhood $V$ of $p_0$ in $\R^n$ such that 
	$p_0=0$, $\lambda^{-1}(0)\cap V=\{(x^\prime,x_n)\in V\mid x_n=0\}$ and
	\begin{align*}
	\frac{\partial^j \lambda}{\partial x^j_n}(0)&=0,\qquad 0\leq j<k,\\
	\frac{\partial^k \lambda}{\partial x^k_n}(0)&\neq 0.
	\end{align*}
	
	Similarly to above, using Lemma \ref{HadamardLemma} we conclude, if we shrink $V$, 
	that there is $\tilde{\lambda}\in\E_\M(V)$ with the following properties:
	$\tilde{\lambda}(x)\neq 0$ and $\lambda(x)=x_n^k\tilde{\lambda}(x)$ for all points
	$x\in V$.
	There is also a Denjoy-Carleman function $f_1$ on $V$ such that $f(x^\prime,x_n)=x_nf_1(x^\prime,x_n)$. 
	We want to show, as in the $1$-dimensional case, that $f_1(x^\prime,0)=0$ for $(x^\prime,0)\in V$ if $k>1$: 
	Suppose that there exists some $y\in\R^{n-1}$ with $(y,0)\in V$ and $f_1(y,0)\neq 0$.
	Then there is a neighbourhood $W$ of $(y,0)$ such that $f_1(x)\neq 0$ and also 
	$\tilde{\lambda}(x)\neq 0$ for $x\in W$. 
	W.l.o.g.\ the open set $W$ is of the form $W=W^\prime\times I\subseteq\R^{n-1}\times\R$ and set
	\begin{equation*}
	F(x_n):=\int_{W^\prime}\biggl\lvert\frac{f_1}{\tilde{\lambda}}(x)\biggr\rvert\,dx
	\end{equation*}
	for $x_n\in I$. We conclude that 
	\begin{equation*}
	\int_W\!\lvert u(x)\rvert\,dx=\int_I\lvert x_n\rvert^{1-k}F(x_n)\,dx=\infty
	\end{equation*}
	and hence $u$ cannot be locally integrable near $(y,0)$ which contradicts our assumption.
	Therefore we obtain by iteration a function $\tilde{f}$ of class $\{\M\}$ defined near the origin in $\R^n$ 
	such that $f(x^\prime,x_n)=x_n^k\tilde{f}(x^\prime,x_n)$.
	Hence $u=\tilde{f}/\tilde{\lambda}$ is also of class $\{\M\}$ in a neighbourhood of $0$.
\end{proof}
\begin{Cor}\label{DivCor}
	Let $U\subseteq\R^n$ a neighbourhood of $0$, $\lambda\in\E_\M(U)$ and suppose that $\lambda$ is of the form
	$\lambda (x)=x^\alpha \tilde{\lambda}(x)$ where $\alpha\in\N_0^n$ and $\tilde{\lambda}\in\E_\M(U)$ with
	$\tilde{\lambda}(0)\neq 0$.
	
	If $u$ is a locally integrable function near $0$ with the property that the product $f:=\lambda\cdot u$ is of
	class $\{\M\}$ near the origin, then $u$ is also ultradifferentiable near $0$.
\end{Cor}
\begin{proof}
	Note first that, if $\alpha=\alpha_je_j$ then the statement is just Proposition \ref{DivProp}. In the general case
	we argue as follows: Set $\tilde{f}=f/\tilde{\lambda}$ and 
	\begin{equation*}
	u_k(x)=\prod_{j=k+1}^n x_j^{\alpha_j} u(x)
	\end{equation*}
	for all $k\in\{1,\dotsc,n-1\}$.
	The function $\tilde{f}$ is of class $\{\M\}$ whereas the functions $u_k$ are locally integrable near $0$. 
	Furthermore we define $u_{n}=u$ and obtain
	\begin{align*}
	x_1^{\alpha_1}u_1(x)&=\tilde{f}(x)\\
	x_{k+1}^{\alpha_{k+1}}u_{k+1}(x)&=u_{k}(x)\qquad 1\leq k\leq n-1.
	\end{align*}
	Hence repeated application of Proposition \ref{DivProp} finishes the proof.
\end{proof}
In the literature the focus regarding questions of divisibility of functions seems to be more on the problem
if it is possible to show that functions that are formally divisible, i.e.\ their Taylor series are divisible, are
actually divisible.
Indeed, the  Weierstrass division theorem for example implies that two real-analytic functions that 
are formally divisible are also divisible as functions.

However, the equivalent of the Weierstrass division theorem does not hold for general quasianalytic 
Denjoy-Carleman classes \cite{MR3190431},\cite{MR3239125}, c.f.\ also \cite{MR2380000}.
In general the algebraic structure of quasianalytic Denjoy-Carleman classes is far more complicated than
that of the space of real-analytic functions, c.f.\ the survey \cite{MR2384272}.

Despite this there are some positive results known for quasianalytic regular classes, 
e.g.\ \cite{MR2061220} showed that certain desingularization theorems hold in these classes
whereas \cite{MR1992825} proved that quasianalytic regular Denjoy-Carleman classes define o-minimal structures. 
Both of these approaches can be used to prove division theorems. 
Especially the following result was shown by \cite{MR3158580}.
\begin{Thm}\label{FormalDiv}
	Let $p\in\R^n$, $\M$ quasianalytic and $f,g\in\E_\M$ are defined near $p$ with power series expansions $\hat{f}$ and $\hat{g}$ at $p$. 
	If $\hat{f}\in\hat{g}\cdot \C[[x]]$ then $f\in g\cdot\E_\M$ near $p$.
\end{Thm}

\section{Almost analytic extensions and the wavefront set in the ultradifferentiable setting}\label{UltraWF}
In this section we recall the almost analytic extension of ultradifferentiable functions given by Dyn'kin in \cite{zbMATH03751341,MR0587795} 
and its connection with the ultradifferentiable wavefront set
introduced by H\"ormander in \cite{MR0294849} that was proven in \cite{Fuerdoes1}.

We recall (see e.g.\ \cite{MR597145}) that a smooth function $F$ given on 
an open subset $\tilde{\Omega}\subseteq\C^n$ is almost analytic iff
\begin{equation*}
\bar{\partial}_jF=\frac{\partial}{\partial \bar{z}_j}F=\frac{1}{2}\biggl(\frac{\partial}{\partial x_j}+i\frac{\partial}{\partial y_j}\biggr)F
\end{equation*}
is flat on $\tilde{\Omega}\cap\R^n$.
The motivation to consider almost analytic function in the ultradifferentiable 
setting is the well-known fact that a function $f$ is smooth on $\Omega$ if and only
if there is an almost analytic function $F$ on some open set $\tilde{\Omega}\subseteq\C^n$ with $\tilde{\Omega}\cap\R^n=\Omega$ such that
$F\vert_\Omega=f$.
In the ultradifferentiable category the idea is now that if $f$ is ultradifferentiable 
of class $\{\M\}$
then it should be possible to construct an almost analytic extension $F$ of $f$ such
that the decrease of $\bar{\partial}_jF$ can be measured in terms of the weight sequence $\M$.
 (c.f.\ \cite{MR1253229}). 

In order to specify this decay we introduce for a regular weight sequence $\M$  its  associated weight given by
\begin{equation}\label{Def:weight}
h_\M (t)=\inf_k t^km_k \quad \text{if } t>0\quad \text{\&}\quad h_\M(0)=0.
\end{equation}
Conversely it is possible to extract the weight sequence from its weight:
\begin{equation*}
m_k=\sup_t\frac{h_\M(t)}{t^k}
\end{equation*}
%In order to describe the connection between the weight and the weight function associated to a regular weight sequence we set
%\begin{align*}
%\tilde{\omega}_\M(t)&:=\sup_{j\in\N_0}\log\frac{t^j}{m_j}\\
%\tilde{h}_\M (t)&=\inf_k t^kM_k 
%\end{align*}
%for $t>0$ and $\tilde{\omega}_\M(0)=\tilde{h}_\M(0)=0$.
The weight $h_\M$ is continuous with values in $[0,1]$, equals $1$ on $[1,\infty)$
and goes more rapidly to $0$ than $t^p$ for any $p>0$ for $t\rightarrow 0$, c.f.\ \cite{Fuerdoes1}.
Before we are able to state the Theorem of Dyn'kin alluded above, we have to note
that his result gives not the existence of a global extension as in the smooth case 
but only a semiglobal statement. This corresponds with the fact that real-analytic functions have generally only local holomorphic extensions.
In order to state the precise form of Dyn'kin's result we recall the following definition from e.g.\ \cite{MR0320743}. 
If $K\subseteq\Omega$ is compact then $\E_{\M}(K)$ is the space of smooth functions which are defined on 
some neighbourhood of $K$ and on $K$ they satisfy \eqref{DCdefiningEst} for
some constants $C,h>0$.
\begin{Thm}\label{Dynkin1}
	Let $\M$ be a regular weight sequence, $K\subset\subset\R^n$ a compact and convex set with 
	$K=\overline{K^\circ}$.% and $f\in\E(\Omega)$. 
	Then $f\in\E_\M(K)$ if and only  if there exists a test function $F\in\D(\C^n)$ with $F|_K=f$ and 
	if there are constants $C,Q>0$ such that
	\begin{equation}
	\bigl\lvert\bar{\partial}_jF(z,\bar{z})\bigr\rvert\leq C h_\M(Qd_K(z))
	\end{equation}
	where $1\leq j\leq n$ and $d_K$ is the distance function with respect to $K$ on $\C^n\!\setminus\!K$.
\end{Thm}
%We shall note that Dyn'kin used the function $h_1(t)=\inf_{k\in\N} m_k t^{k-1}$ instead of the weight $h_\M$\footnote{$h_1$ is in fact the weight associated to the shifted sequence $(m_{k+1})_{k}$}.
%It is easy to see that $h_1(t)=h_\M(t)/t$. 
%But we observe that 
%\begin{equation*}
%h_\M(t)=\inf_{k\in\N_0}m_kt^k\leq t\inf_{k\in\N}m_k t^{k-1}=th_1(t)\leq Ct\inf_{k\in\N}m_{k-1}t^{k-1}=Cth_{\M}(t).
%\end{equation*}
%Since $h_\M$ is rapidly decreasing for $t\rightarrow 0$ we can 
%interchange these two functions in the formulation of Theorem \ref{Dynkin1}. 
%In fact, Dyn'kin's proof gives immediately the following result.
%We state here only the partial result we need later.
The local form of Theorem \ref{Dynkin1} is 
\begin{Cor}\label{CharMalmostanalytic}
	If $f$ is ultradifferentiable of class $\{\M\}$ near $p$, 
	%i.e.\ there exists a compact neighbourhood $K$ of $p$ such that $f\vert_K\in\E_\M(K)$,
	%\eqref{DC-classDef} holds
	then  there are an open neighbourhood $W\subseteq\Omega$,
	a constant $\rho>0$ and a function $F\in\E(W+iB(0,\rho))$ such that $F|_W=f|_W$ and
	\begin{equation}\label{Malmostest}
	\bigl\lvert\bar{\partial}_j F(x+iy)\bigr\rvert\leq Ch_\M (Q\lvert y\rvert)
	\end{equation}
	for some positive constants $C,Q$ and all $1\leq j\leq n$ and $x+iy\in W+iB(0,\rho)$.
\end{Cor}
We call such function $F$ an $\M$-almost analytic extension of $f$.

The following theorem is the $\M$-almost analytic version of the almost-holomorphic implicit function 
theorem proven in \cite{MR2085046}.
\begin{Thm}\label{Lamel-implicit}
	Let $\M$ be a regular weight sequence, $U\subseteq\C^N$ a neighbourhood of the origin, $A\in\C^p$ 
	and $F:U\times\C^p\rightarrow \C^N$
	of class $\{\M\}$ on $U$ and polynomial in the last variable with $F(0,A)=0$ and $F_Z(0,A)$ invertible.
	Then there exists a neighbourhood $U^\prime\times V^\prime$ of $(0,A)$ and a smooth mapping
	$\phi=(\phi_1,\dotsc,\phi_N): U^\prime\times V^\prime\rightarrow \C^N$ with $\phi(0,A)=0$ with the property that if 
	$F(Z,\bar{Z},W)=0$ for some $(Z,W)\in U^\prime\times V^\prime$ then $Z=\phi (Z,\bar{Z},W)$.
	Furthermore, there are constants $C,\gamma>0$ such that 
	\begin{equation}\label{AlmostholomorphicEst}
	\biggl\lvert\frac{\partial \phi_j}{\partial Z_k}(Z,\bar{Z},W)\biggr\rvert \leq Ch_\M\bigl(\gamma\lvert \phi(Z,\bar{Z},W)- Z\rvert\bigr)
	\end{equation}
	for all $1\leq j,k\leq N$ and $\phi$ is holomorphic in $W$.
\end{Thm}
\begin{proof}
	We write $F(Z,\bar{Z},W)=F(x,y,W)$, where $(x,y)\in\R^N\times\R^N$ are the underlying real 
	coordinates of $\C^N$, i.e.\ $Z_j=x_j+iy_j$ for $1\leq j\leq N$. 
	Let $U_0\subseteq\R^N$ be a convex neighbourhood of $0$ such that $\overline{U_0\times U_0}\subseteq U$.
	Using Theorem \ref{Dynkin1} we find a smooth mapping
	\begin{equation*}
	\tilde{F}=U_0\times\R^N\times U_0\times \R^N\times\C^p\longrightarrow \C^N
	\end{equation*}
	such that $\tilde{F}(x,x^\prime,y,y^\prime,W)\big\vert_{x^\prime=y^\prime=0}=F(x,y,W)$ and 
	if we write $\xi_k=x_k+ix^\prime_k$, $\eta_k=y_k+iy^\prime_k$ for $k=1,\dotsc,N$ and set
	$\zeta =(\xi,\eta)$, then for each compact subset $K\subset\subset \C^p$ 
	there are constants $C,\gamma >0$ such that
	\begin{subequations}\label{Dynkin2}
		\begin{align}
		\biggl\lvert \frac{\partial\tilde{F}_j}{\partial \bar{\xi}_k}(\zeta,\bar{\zeta},W)\biggr\rvert 
		&\leq Ch_\M(\gamma\lvert\imag\zeta\rvert)\\
		\biggl\lvert \frac{\partial\tilde{F}_j}{\partial \bar{\eta}_k}(\zeta,\bar{\zeta},W)\biggr\rvert 
		&\leq Ch_\M(\gamma\lvert\imag\zeta\rvert)
		\end{align}
	\end{subequations}
	for $(\zeta,W)\in (U_0+i\R^N)^2\times K$ and $1\leq j,k\leq N$. 
	Note also that $\tilde{F}$ is still polynomial in the variable $W$.
	
	We introduce new variables $\chi=(\chi_1,\dots,\chi_N)\in\C^N$ by
	\begin{align*}
	\xi_k &=\frac{Z_k+\chi_k}{2} & \eta_k&=\frac{Z_k-\chi_k}{2i} & 1&\leq k\leq N\\
	\intertext{and note that}
	x_k&=\frac{Z_k+\chi_k}{2}\bigg\vert_{\chi_k=\bar{Z}_k} & y_k&=\frac{Z_k-\chi_k}{2i}\bigg\vert_{\chi_k=\bar{Z}_k}.
	\end{align*}
	
	We also set $G(Z,\bar{Z},\chi,\bar{\chi},W)=\tilde{F}(\xi,\bar{\xi},\eta,\bar{\eta},W)$.
	The function $G$ is therefore smooth in the first $2N$ variables in some neighbourhood of the origin
	and polynomial in the last $p$ variables. 
	Due to the definition of $G$ %and the  properties of $\tilde{F}$ 
	we have
	\begin{align*}
	\frac{\partial G}{\partial \bar{Z}}&=\frac{1}{2}\frac{\partial \tilde{F}}{\partial \bar{\xi}}
	+\frac{1}{2i}\frac{\partial \tilde{F}}{\partial \bar{\eta}}\\
	\frac{\partial G}{\partial \bar{\chi}}&=\frac{1}{2}\frac{\partial \tilde{F}}{\partial \bar{\xi}}
	-\frac{1}{2i}\frac{\partial \tilde{F}}{\partial \bar{\eta}}.
	\end{align*}
	
	We are going to compute the real Jacobian of $G$
	at the point $(0,A)$. We obtain
	\begin{align*}
	\frac{\partial G}{\partial Z}(0,A)&=\frac{\partial F}{\partial Z}(0,A)\\ 
	\intertext{and}
	\frac{\partial G}{\partial \bar{Z}}(0,A)&=\frac{1}{2}\biggl(\frac{\partial \tilde{F}}{\partial \bar{\xi}}(0,A)
	-i\frac{\partial\tilde{F}}{\partial \bar{\eta}}(0,A)\biggr)=0
	\end{align*}
	and thus
	\begin{equation*}
	\det \begin{pmatrix}
	\frac{\partial G}{\partial Z} &  \frac{\partial G}{\partial \bar{Z}}\\
	\frac{\partial \bar{G}}{\partial Z} &  \frac{\partial \bar{G}}{\partial \bar{Z}}
	\end{pmatrix} (0,A)
	=\biggl\lvert\,\det \frac{\partial F}{\partial Z}(0,A)\,\biggr\rvert^2\neq 0
	\end{equation*}
	by assumption. Hence, by the smooth implicit function theorem, there is a smooth mapping $\psi$ defined 
	in some open neighbourhood of $(0,A)$, valued in $\C^N$ and holomorphic in the variable $W$
	such that $Z=\psi (\chi,\bar{\chi},W)$ solves the equation $G(Z,\bar{Z},\chi,\bar{\chi},W)=0$ uniquely.
	Since $G(Z,\bar{Z},\bar{Z},Z,W)=F(Z,\bar{Z},W)$, this fact implies that if $F(Z,\bar{Z},W)=0$ then
	$Z=\psi(\bar{Z},Z,W)$.
	We set $\phi (Z,\bar{Z},W)=\psi(\bar{Z},Z,W)$ and claim that $\varphi$ satisfies 
	\eqref{AlmostholomorphicEst}. 
	
	In fact, if we differentiate the implicit equation 
	$G(\psi(\chi,\bar{\chi},W),\overline{\psi(\chi,\bar{\chi},W)},\chi,\bar{\chi},W)=0$ then we obtain
	\begin{align*}
	G_Z\psi_{\bar{\chi}}+G_{\bar{Z}}\bar{\psi}_{\bar{\chi}}+G_{\bar{\chi}}&=0\\
	\bar{G}_{\bar{Z}}\bar{\psi}_{\bar{\chi}}+\bar{G}_Z\psi_{\bar{\chi}}+\bar{G}_{\bar{\chi}}&=0.
	\end{align*}
	If we multiply the last line with $G_{\bar{Z}}\bar{G}_{\bar{Z}}^{-1}$ 
	and substract the result from the first line then 
	\begin{equation*}
	\bigl(G_Z -G_{\bar{Z}}\bar{G}_{\bar{Z}}^{-1}\bar{G}_Z\bigr)\psi_{\bar{\chi}}
	=G_{\bar{Z}}\bar{G}_{\bar{Z}}^{-1}\bar{G}_{\bar{\chi}}-G_{\bar{\chi}}.
	\end{equation*}
	
	Hence we have in a small neighbourhood of $(0,A)$ that
	\begin{equation*}
	\phi_Z(Z,\bar{Z},W)=\psi_{\bar{\chi}}(\bar{Z},Z,W)
	=\Biggl(\frac{G_{\bar{Z}}\bar{G}_{\bar{Z}}^{-1}
		\bar{G}_{\bar{\chi}}-G_{\bar{\chi}}}{G_Z -G_{\bar{Z}}\bar{G}_{\bar{Z}}^{-1}\bar{G}_Z}\Biggr)
	\Bigl(\psi(\bar{Z},Z,W),\overline{\psi(\bar{Z},Z,W)},\bar{Z},Z,W\Bigr).
	\end{equation*}
	This formula shows that any function $\partial_{Z_k}\varphi_{j}$ is a sum of products each of which 
	contains a factor of the form $G_{\bar{Z}_\ell}$ or $G_{\bar{\chi}_\ell}$ for some $\ell$.
	Note also that by definition $\imag \xi=\tfrac{1}{2}(\imag Z+\imag\chi)$ and 
	$\imag\eta=-\tfrac{1}{2}(\real Z-\real\chi)$.
	
	Hence \eqref{Dynkin2} implies on some compact neighbourhood of $(0,A)$, where $\det G_Z^{-1}$ is bounded,
	\begin{equation*}
	\begin{split}
	\bigl\lvert\phi_Z(Z,\bar{Z},W)\bigr\rvert 
	&\leq Ch_\M \biggl(\frac{1}{2}\gamma \bigl(\lvert \imag \phi(Z,\bar{Z},W) 
	-\imag Z\lvert^2+ \lvert \real Z -\real \phi(Z,\bar{Z},W)\rvert^2\bigr)^{\tfrac{1}{2}}\biggr)\\
	&=C h_\M\bigl(\gamma \lvert \phi(Z,\bar{Z},W)-Z\rvert\bigr)
	\end{split}
	\end{equation*}
	for some positive constants $C$ and $\gamma$.
\end{proof}

In the following we recall the results on the ultradifferentiable wavefront set 
that we need in this paper. We start with the definition given in \cite{MR0294849}.
%In 1971 H{\"o}rmander \cite{MR0294849} proved the following local characterization of $\E_\M$ via the Fourier transform:
%\begin{Prop}\label{MCharFT}
%	Let $u\in\D^\prime (\Omega)$ and $p_0\in\Omega$. 
%	Then $u$ is ultradifferentiable of class $\{\M\}$ near $p_0$ 
%	if and only if there are an open neighbourhood $V$ of $p_0$, 
%	a bounded sequence $(u_N)_N\subseteq\E^\prime (U)$
%	such that $u|_V=(u_N)|_V$ and some constant $Q>0$ so that
%	\begin{equation*}
%	\sup_{\substack{\xi\in\R^n\\ 
%			N\in\N_0}}\frac{\lvert\xi\rvert^N\lvert \hat{u}_N(\xi)\rvert}{Q^N M_N}<\infty.
%	\end{equation*}
%\end{Prop}
%Subsequently he used this fact to define analogously to the smooth category:
\begin{Def}\label{WF-M Def1}
	Let $u\in\D^\prime (\Omega)$ and $(x_0,\xi_0)\in T^*\Omega\!\setminus\!\{ 0\}$.
	We say that $u$ is \emph{microlocally ultradifferentiable of class $\{\M\}$} at $(x_0,\xi_0)$ 
	iff there is a bounded sequence 
	$(u_N)_N\subseteq\E^\prime (\Omega)$ such that $u_N\vert_V\equiv u\vert_V$, where $V\in\U(x_0)$
	and a conic neighbourhood $\Gamma$ of $\xi_0$  such that for some constant $Q>0$
	\begin{equation}\label{WF-M Estimate1}
	\sup_{\substack{\xi\in\Gamma\\
			N\in\N_0}} \frac{\lvert\xi\rvert^N\lvert\hat{u}_N\rvert}{Q^N m_N N!}<\infty.
	\end{equation}
	The ultradifferentiable wavefront set $\WF_\M u$ is then defined as
	\begin{equation*}
	\WF_\M u:=\bigl\{(x,\xi)\in T^*\Omega\!\setminus\!\{ 0\}\mid u\text{ is not microlocal of class }\{\M\}
	\text{ at }(x,\xi)\bigr\}.
	\end{equation*}
\end{Def}
The basic properties $\WF_\M$ shown by H{\"o}rmander in \cite{MR1996773} are the following.
\begin{Thm}[\cite{MR1996773} Theorem 8.4.5-8.4.7]\label{WF-MProperties}
	Let $u\in\D^{\prime}(\Omega)$ and $\M,\mathcal{N}$ weight sequences. Then we have 
	\begin{enumerate}
		\item $\WF_\M u$ is a closed conic subset of $\Omega\times\R^n\!\setminus\!\{0\}$.
		\item The projection of $\WF_\M u$ in $\Omega$ is
		\begin{equation*}
		\pi_1\bigl(\WF_\M u\bigr)=\mathrm{sing}\,\supp_{\M} u=
		\overline{\bigl\{x\in\Omega \;\vert\; \nexists V\in\U(x):\, u\vert_V\in\E_\M (V)\}}
		\end{equation*}
		\item $\WF u\subseteq \WF_{\NN}u\subseteq\WF_{\M}u$ if $\M\preccurlyeq \mathcal{N}$.
		\item If $P=\sum p_{\alpha}D^{\alpha}$\footnote{We use in the following the notation $D_j=-i\partial_j$.} is a partial differential operator with ultradifferentiable coefficents
		of class $\{\M\}$ then $\WF_{\M} Pu\subseteq \WF_\M u$.
	\end{enumerate}
\end{Thm}
Additionally we note that $\WF_\M u$ satisfies the following \emph{microlocal reflection property}:
\begin{equation}\label{microreflprop}
(x,\xi)\notin\WF_\M u \Longleftrightarrow (x,-\xi)\notin\WF_\M \bar{u}
\end{equation}
In particular, if $u$ is a real-valued distribution, i.e.\ $\bar{u}=u$, then 
$\WF_\M u\vert_{x}:=\{\xi\in\R^n\mid (x,\xi)\in\WF_\M u\}$ is symmetric at the origin.
%\begin{Ex}
%	It is easy to see that $\WF_\M \delta_p=\{p\}\times\R^n\!\setminus\!\{0\}$ 
%	for any regular weight sequence $\M$.
%\end{Ex}

%Our aim in this section is to develop, using the almost-analytic extension of functions in $\E_\M$ given by Dyn'kin, 
%a geometric description of $\WF_\M$ similarly to the one that was presented e.g.\
%by Liess \cite[section 4]{MR1806500} for the smooth wavefront set. 

It is a classic fact that the analytic wavefront set can not only characterized by
the Fourier transform but also holomorphic extension in certain directions, see
\cite{MR0650834}. Likewise, the smooth wavefront set can be characterized by
almost-analytic extensions, c.f.\ \cite{MR1806500}. 
We present now the basic results on the connection between almost-analytic extensions
and the ultradifferentiable wavefront set that we proved in \cite{Fuerdoes1}.
In order to do so we need first to recall some notations used in \cite{Fuerdoes1}: 
A subset $\Gamma\subseteq\R^d$ is a cone iff for all $\lambda>0$ 
and $y\in\Gamma$ we have $\lambda y\in\Gamma$.
If $r>0$ then 
\begin{equation*}
\Gamma_r :=\bigl\{y\in \Gamma \mid\, \lvert y\rvert <r\bigr\}.
\end{equation*} 
If $\Gamma^{\prime}\subseteq\Gamma$ is also a cone we write 
$\Gamma^{\prime}\subset\subset\Gamma$ iff 
$(\Gamma^{\prime}\cap S^{d-1})\subset\subset (\Gamma\cap S^{d-1})$.

%Analogous to Liess \cite[section 2.1]{MR1806500} (c.f.\ also Lamel \cite[section 2]{MR2085046}) in the smooth category 
If $\M$ is a weight sequence with associated weight $h_\M$ then a function 
$F\in\E(\Omega\times U\times\Gamma_r)$, $U\subseteq\R^d$ open, is said to be
\emph{$\M$-almost analytic} in the variables $(x,y)\in U\times \Gamma_r$ with parameter 
$x^\prime\in\Omega$ iff for all  $K\subset\subset\Omega$, 
$L\subset\subset U$ and cones $\Gamma^{\prime}\subset\subset\Gamma$ there are constants 
$C,Q>0$ such that for some $r^\prime$ we have
\begin{equation}
\biggl\lvert\frac{\partial F}{\partial \bar{z}_j}(x^\prime,x,y)\biggr\rvert \leq Ch_\M (Q\lvert y\rvert) 
\qquad (x^\prime,x,y)\in K\times L\times\Gamma^{\prime}_{r^\prime},\;j=1,\dotsc,d
\end{equation}
where $\tfrac{\partial}{\partial \bar{z}_j}=\tfrac{1}{2}(\partial_{x_j}+i\partial_{y_j})$.
%and $h_\M$ is  the weight associated to the regular weight sequence $\M$ as defined by \eqref{Def:weight}.

We may also say generally that a function $g\in\CC (\Omega\times U\times\Gamma_r)$ 
is of \emph{slow growth} in $ y\in\Gamma_r$ 
if for all $K\subset\subset\Omega$, $L\subset\subset U$ and $\Gamma^\prime\subset\subset\Gamma$
there are constants $c,k>0$ such that
\begin{equation}\label{temperategrowth}
\lvert g(x^\prime,x,y)\rvert\leq c \lvert y\rvert ^{-k} \qquad (x^\prime,x,y)\in K\times L\times
\Gamma^{\prime}_r.
\end{equation}
%The next theorem is a generalization of \cite[Theorem 4.4.8]{MR1996773}.
\begin{Thm}\label{Theorem-M-BVWF}
	Let $F\in \E(\Omega\times U\times \Gamma_r)$ be $\M$-almost analytic in the variables $(x,y)\in U\times\Gamma_r$
	and of slow growth in the variable $y\in\Gamma_r$.
	Then the distributional limit $u$ of the sequence 
	$u_\eps =F(\,.\,,\,.\,,\eps)\in\E(\Omega\times U)$
	exists. We say that $u=b_\Gamma (F)\in\D^\prime(\Omega\times U)$ is the boundary value of $F$.
	Furthermore, we have
	\begin{equation*}
	\WF_\M u\subseteq\,\bigr(\Omega\times U\bigr)\times\bigl(\R^n\times\Gamma^\circ\bigr)
	\end{equation*}
	where $\Gamma^\circ=\{\eta\in\R^d\mid \langle y,\eta\rangle\geq 0 \;\;\forall y\in\Gamma\}$
	is the dual cone of $\Gamma$ in $\R^d$.
\end{Thm}
%\begin{Ex}\label{ExJump}
%Consider the holomorphic function $F(z)=\tfrac{1}{z}$ on $\C\!\setminus\!\{0\}$. 
%It is well known that the boundary values of $F$ onto the real line from above and beneath, commonly denoted by
%\begin{align*}
%\frac{1}{x+i0}&=b_+F=\lim_{y\rightarrow 0+}\frac{1}{x+iy}\\
%\frac{1}{x-i0}&=b_+F=\lim_{y\rightarrow 0+}\frac{1}{x-iy}\\
%\intertext{satisfy the jump relations (c.f.\ e.g.\ Duistermaat-Kolk \cite{MR2680692}), in particular}
%2i\delta &=\frac{1}{x-i0}-\frac{1}{x+i0}.
%\end{align*}
%We have that both $\tfrac{1}{x+i0}$ and $\tfrac{1}{x-i0}$ are real-analytic outside the origin.
%Hence the application of Theorem \ref{Theorem-M-BVWF} together with the jump relations imply that
%\begin{equation*}
%\WF_\M \biggl(\frac{1}{x\pm i0}\biggr)=\{0\}\times\R_\pm.
%\end{equation*}
 %\end{Ex}

%There is a partial converse to the last theorem.
\begin{Thm}\label{BV-M-WF}
Let $\Gamma\subset \R^n$ be an open convex cone and $u\in\D^\prime(\Omega)$
with $\WF_\M u\subseteq\Omega\times\Gamma^\circ$. 
If $V\subset\subset\Omega$ and $\Gamma^\prime$ is an open convex cone 
with $\overline{\Gamma}^\prime\subseteq\Gamma\cup\{0\}$ then there is an $\M$-almost analytic function $F$ on 
$V+i\Gamma^\prime_r$ of slow growth for some $r>0$ such that $u\vert_V=b_{\Gamma^\prime}(F)$
\end{Thm}
Using Theorem \ref{Theorem-M-BVWF} and Theorem \ref{BV-M-WF} we were able in \cite{Fuerdoes1} to show the characterization of the ultradifferentiable wavefront
set by $\M$-almost analytic extensions.
\begin{Cor}\label{WF-M-localdescr1}
	Let $u\in\D^{\prime}(\Omega)$ and $(x_0,\xi_0)\in\Omega\times\R^n\!\setminus\!\{0\}$. 
	Then $(x_0,\xi_0)\notin\WF_\M u$ if and only if there are a neighbourhood $V$ of $x_0$,
	open convex cones $\Gamma_1,\dots,\Gamma_N$ with the properties 
	$\xi_0\Gamma_j<0$, $j=1,\dots N$ and $\Gamma_j\cap\Gamma_k=\emptyset\text{ for}\;j\neq k$,
	and $\M$-almost analytic functions $h_j$ on $V+i\Gamma_{r_j}$, $r_j>0$, of slow growth
	such that
	\begin{equation*}
	u\vert_V=\sum_{j=1}^N b_{\Gamma_j}(h_j)
	\end{equation*}
\end{Cor}
In \cite{Fuerdoes1} Corollary \ref{WF-M-localdescr1} is then applied to show the following Theorem.
\begin{Thm}
	Let $F$ be an $\E_\M$-diffeomorphism then
	\begin{equation*}
	\WF_\M \bigl(F^*u\bigr) =F^*\bigl(\WF_\M u\bigr).
	\end{equation*}
\end{Thm}
	Hence if $M$ is an $\E_\M$-manifold and $u\in\D^{\prime}(M)$ we can define $\WF_{\M} u$ invariantly
	as a subset of $T^{\ast}M\!\setminus\!\{0\}$, c.f.\ \cite{Fuerdoes1}.
	We refer to \cite{MR0516965} or \cite{MR678605} for the definition of distributions on manifolds, either scalar or with values in vector bundles.
	Let $u$ be a distribution on an ultradifferentiable manifold $M$ of class $\{\M\}$ 
with values in an $\E_{\M}$-vector bundle over $M$. 
In particular we can write locally $u\vert_V=\sum_{j=1}^N u_j\omega^j$,
where $V\subseteq M$ is an open subset of $M$, scalar-valued distributions $u_j\in\D^\prime(V)$ and the sections $\omega^1,\dotsc,\omega^N\in\E_{\M}(V,E\vert_V)$
constitute a local basis of $\E_\M(M,E)$.
The ultradifferentiable wavefront set of $u$ is then defined locally by
\begin{equation*}
\WF_{\M}u=\bigcup_{j=1}^N \WF_{\M}u_j.
\end{equation*}

We close this section by recalling the last fact that we need from \cite{Fuerdoes1},
the elliptic regularity theorem for partial differential operators 
with ultradifferentiable coefficients.
In order to state it correctly we have to recall again some notations from 
\cite{Fuerdoes1}, for more details on the constructions see \cite{MR678605}.
To begin with if 
\begin{equation*}
Q(x,D)=\sum_{\lvert\alpha\rvert\leq m}q_\alpha(x)D^\alpha
\end{equation*}
is a partial differential opertator on $\Omega$, i.e.\ $q_\alpha\in\E(\Omega)$, of order $\leq m$ then its principal symbol
\begin{equation*}
q(x,\xi)=\sum_{\lvert\alpha\rvert=m}q_\alpha(x)\xi^\alpha
\end{equation*}
is a smooth function on $T^\ast\Omega$ that is homogeneous of degree $m$ in the second variable.
Let $M$ be an $\E_{\M}$-manifold and $E$ and $F$ two ultradifferentiable 
vector bundles of class $\{\M\}$ over $M$ with the same fiber dimension $\nu$.
An ultradifferentiable differential operator  
$P:\;\E_\M(M, E)\rightarrow\E_\M(M,F)$ of class $\{\M\}$ and order $\leq m$
is given locally in some trivialization by
\begin{equation*}
Pu=
\begin{pmatrix}
P_{11} &\cdots & P_{1\nu}\\
\vdots & \ddots & \vdots\\
P_{\nu 1}&\cdots & P_{\nu\nu}
\end{pmatrix}
\begin{pmatrix}
u_1\\
\vdots\\
u_\nu
\end{pmatrix},
\end{equation*}
where the $P_{jk}$ are partial differential operators with ultradifferentiable
coefficients of order $\leq m$ defined on some chart neighbourhood.
The operator $P$ is of order $m$ if it is not of order $\leq m-1$.
The principal symbol $p$ of $P$ is an ultradifferentiable mapping on $T^\ast M$
with values in the fiber-linear maps from $E$ to $F$, that is given locally by
\begin{equation*}
p(x,\xi)=\begin{pmatrix}
p_{11}(x,\xi)&\dots&p_{1\nu}(x,\xi)\\
\vdots&\ddots &\vdots\\
p_{\nu 1}(x,\xi)&\dots &p_{\nu\nu}(x,\xi)
\end{pmatrix}
\end{equation*}
where $p_{jk}$ is the principal symbol of $P_{jk}$.
The operator $P$  is not characteristic (or non-characteristic) 
at a point $(x,\xi)\in T^\ast M\!\setminus\!\{0\}$ if $p(x,\xi)$ 
is an invertible  linear mapping. The set of all characteristic points
is defined by
\begin{equation*}
\Char P=\{(x,\xi)\in T^{\ast}M\!\setminus\!\{0\}\,\colon P\text{ is characteristic at }(x,\xi)\}.
\end{equation*}
After this lengthy preparation we are able to state the elliptic regularity theorem 
for  partial differential operators between ultradifferentiable 
vector bundles.
\begin{Thm}\label{elliptic-regThm}
	Let $M$ be an $\E_{\M}$-manifold and $E,F$ two ultradifferentiable vector bundles on $M$ of the same 
	fiber dimension.
	If $P(x,D)$ is a differential operator between $E$ and $F$ with $\E_\M$-coefficients
	and $p$ its principal symbol, then
	\begin{equation}\label{elliptic-regEq}
	\WF_{\M}u\subseteq \WF_\M (Pu)\cup \mathrm{Char\,}P\qquad u\in\D^{\prime}(M,E).
	\end{equation}
\end{Thm}
\section{CR Manifolds of Denjoy Carleman type}\label{sec:ultraCR}
In this section we rapidly recall the basic definitions of CR geometry, for more details see \cite{MR1668103}.
We begin with the embedded case. Let $M\subseteq\C^N$ be a real submanifold of $\C^N$,
then $T_pM\subseteq T_p\C^N$ ($p\in M$) as real vector spaces, but $T_p\C^N=\R^{2N}\cong \C^N$ inherits
a complex structure from $\C^N$. Hence there is a maximal complex subspace $T^c_pM$ of $T_p\C^N$ such that $T_p^cM\subseteq T_pM\subseteq T_p\C^N$.

\begin{Def}
	A submanifold $M\subseteq\C^N$ is said to be CR if the mapping 
	\begin{equation*}
	M\ni p\longmapsto \dim_\C T^c_pM
	\end{equation*}
	is constant. The CR dimension of $M$ is then defined as $\dim_{C\! R} M:=\dim_\C T^c_pM$.
\end{Def}
Note that any real hypersurface $M\subseteq\C^N$ is CR. 
An arbitrary submanifold $M\subseteq \C^N$ of codimension $d$ is said to be
\emph{generic} iff it can be realized as the intersection of $d$ real hypersurfaces 
whose complex tangent spaces are in general position as complex vector spaces. 
The manifold $M$ is said to be generic at a point $p\in M$ iff there is a neighbourhood $U$ of $p$ in $\C^N$
such that $M\cap U$ is generic.
We recall that if $M\subseteq\C^N$ is a generic submanifold of CR dimension $n$ 
and real codimension $d$ then $n+d=N$.

It is easy to see that for a CR manifold $M$ we can consider the complex tangent bundle $T^cM\subseteq TM$.
However the complex tangent bundle, although being a vector bundle over $\C$, is realized as a subbundle of
the real bundle $TM$. Often it is more convenient to take a different approach for the definition of CR manifolds.
For this end consider the complexified tangent bundle $\C TM=\C\otimes TM$ 
of a manifold $M\subseteq \C^N$.
Furthermore let $p\in M$ and set $\C T_p\C^N= T^{1,0}_p\C^N\oplus T^{0,1}_p\C^N$.
If $z_j=x_j+iy_j$, $j=1,\dotsc, N$ denote the  coordinates of $\C^N$ 
then the spaces $T^{1,0}_p\C^N$ and $T^{0,1}_p\C^N$ 
are generated by $\partial /\partial z_j\vert_p$ and
$\partial /\partial \bar{z}_j\vert_p$, $j=1,\dotsc, N$, respectively.
If we set $\crb_p=\C T_pM\cap T^{0,1}_p \C^N$ then $\dim_\C\crb_p=\dim_\C T^c_p M$.
If $M$ is a CR submanifold, then $\crb=\bigsqcup_p\crb_p$ is said to be the CR bundle associated to $M$.
It is easy to see that $\crb$ is involutive, i.e.\ $[\crb,\crb]\subseteq \crb$, and $\crb\cap\bar{\crb}=\{0\}$.
Using this it is possible to generalize the notion of CR manifold.
\begin{Def}
	Let $M$ be a manifold (not necessarily embedded) and $\crb\subseteq \C TM$ a subbundle.
	We say that $(M,\crb)$ is an abstract CR manifold iff $\crb$ is an involutive bundle and
	$\crb\cap\bar{\crb}=\{0\}$. The CR dimension of $M$ is defined as $\dim_{C\! R}M=\dim \crb$.
	If $\dim_\R M=m+n$ then the CR codimension is given by $d=m-n$.
\end{Def}
If $M$ is a CR manifold of class $\{\M\}$ then a CR vector field $L$ is an ultradifferentiable section of $\crb$,
i.e.\ $L\in\E_\M(M,\crb)$.
If $p\in M$ and $n=\dim_{C\! R} M$ then
a local basis of CR vector fields near $p$ consists of $n$ CR vector fields $L_1,\dotsc, L_n$ defined near $p$
that are linearly independent. 
We also set $L^\alpha=L_1^{\alpha_1}\dotsb L_n^{\alpha_n}$
for $\alpha\in\N_0^n$.

A CR function or CR distribution is a function or distribution on $M$ that is annihilated by all CR vector fields.
We refer to $T^\prime M:=\crb^\perp$ as the holomorphic cotangent bundle.
$T^\prime M$ is a complex vector bundle on $M$ with fiber dimension $N=n+d$.
Its ultradifferentiable sections are called holomorphic forms.
The real subbundle $T^0M\subseteq T^\prime M$ that consists of the real dual vectors that vanish on 
$\crb \oplus\bar{\crb}$ is called the characteristic bundle of $M$ 
and its sections of class $\{\M\}$ are the characteristic forms on $M$.
Note that if $L$ is a CR vector field, we have generally that $\Char L\subseteq T^0 M$, hence we obtain 
for any CR distribution $u$ that
$\WF_\M u\subseteq T^0M$.

A $\CC^1$-mapping $H$ between two CR manifolds $(M,\crb)$ and $(M^\prime,\crb^\prime)$ is CR iff
for all $p\in M$ we have $H_\ast(\crb_p)\subseteq \crb^\prime_{H(p)}$.
Here $H_\ast$ denotes the tangent map of $H$. 
If $M^\prime\subseteq\C^{N^\prime}$ is an embedded CR submanifold and 
$Z^\prime=(Z^\prime_1,\dotsc, Z^\prime_{N^\prime})$ some set of local holomorphic coordinates in
$\C^{N^\prime}$ then $H_j= Z^{\prime}_j\circ H$, $1\leq j\leq N^\prime$ is a CR function 
on the CR manifold $M$ for all $1\leq j\leq N^\prime$.

We continue with a first look at specific results about ultradifferentiable CR manifolds.
\begin{Prop}\label{genericcoordinates}
	Let $M\subseteq\C^N$ be a generic manifold of class $\{\M\}$ of codimension $d$ and $p_0\in M$. 
	If $n$ denotes the CR dimension of $M$ then there are holomorphic coordinates 
	$(z,w)\in\C^n\times\C^d$ defined near $p_0$ that vanish 
	at $p_0$ and a function $\varphi\in\E_\M(U\times V,\R^d)$ defined on a neighbourhood 
	$U\times V$ of the origin in $\R^{2n}\times\R^d$ with $\varphi(0)=0$ and $\nabla \varphi(0)=0$,
	such that near $p_0$ the manifold $M$ is given by
	\begin{equation}\label{definingequation}
	\imag w =\varphi (z,\bar{z},\real w).
	\end{equation}
\end{Prop}
\begin{proof}
	We follow the proof in \cite{MR1668103} for the result in the smooth category.
	
	After an affine transformation we may assume that $p_0=0$.
	Let $\rho=(\rho_1,\dotsc,\rho_d)$ be a defining function for $M$ near $0$.  
	The complex differentials $\partial \rho_1,\dotsc,\partial\rho_d$ are linearly independent over $\C$ near $0$ since $M$ is generic.
	For each $k\in\{1,\dotsc,d\}$ we write
	\begin{equation*}
	\rho_k ( Z,\bar{Z})=\sum_{r =1}^N\Bigl( a_{kr}x_r +b_{kr}y_r\Bigr)+ O(2)
	\end{equation*}
	where $O(2)$ denotes terms that vanish at least of quadratic order at $0$.
	Since $\rho_k$ is real-valued, the coefficients $a_{kr}$ and $b_{kr}$ have to be real numbers.
	We define a linear form $\ell_k$ on $\C^N$ by
	\begin{equation*}
	\ell_k(Z)=\sum_{r=1}^N (b_{kr}+ia_{kr})Z_r
	\end{equation*}
	and thus the above equation becomes
	\begin{equation*}
	\rho_k( Z,\bar{Z})=\imag \ell_k(Z) + O(2).
	\end{equation*}
	
	The linear forms $\ell_k$, $k=1,\dotsc,d$ are linearly independent over $\C$ since the differentials 
	$\partial \rho_k$, $k=1,\dotsc,d$, are $\C$-linearly indepedent.
	After renumbering the coordinates $Z_j$ we can assume that
	\begin{equation*}
	Z_1, \dotsc, Z_n,\ell_1,\dotsc,\ell_k
	\end{equation*}
	are linearly indepedent as linear forms over $\C$.
	
	We define new holomorphic coordinates $(z,w)$ near $(0,0)\in\C^{n+d}$ by
	\begin{align*}
	z_j &=Z_j & 1&\leq j\leq n\\
	w_k&=\ell_k (Z) & n+1&\leq k\leq N=n+d.
	\end{align*}
	%We note that the vectors $\ell_k$ are $\C$-linearly independent at $0$ 
	%since the $\partial \rho_k(0)$ are linearly independent over $\C$.
	In these new coordinates we have, if we set $\tilde{\rho}(z,\bar{z},w,\bar{w})=\rho(Z(z,w),\overline{Z(z,w)})$,
	\begin{equation}\label{CRnewcoord1}
	\tilde\rho (z,\bar{z},w,\bar{w})=\imag w + O(2)
	\end{equation}
	and therefore we can locally near $0$ solve the equation
	\begin{equation}\label{CRimplicitcoord}
	\tilde{\rho}(z,\bar{z},w,\bar{w})=0
	\end{equation}
	with respect to $t=\imag w$ according to Theorem \ref{CMStability}.
	We obtain an ultradifferentiable solution $\varphi$ of class $\{\M\}$ defined near 
	$0\in\R^{2n+d}=\C^n\times\R^d$ and valued in $\R^d$. 
	The properties $\varphi(0)=0$ and $\nabla\varphi(0)=0$ are easy consequences of
	\eqref{CRnewcoord1} and \eqref{CRimplicitcoord}. 
	We also see that in view of \eqref{CRnewcoord1} and
	\begin{equation*}
	\tilde{\rho}(z,\bar{z},s +i\varphi(z,\bar{z},s),s -i\varphi(z,\bar{z},s))=0
	\end{equation*}
	the function $\psi(z,\bar{z},s,t)=t-\varphi(z,\bar{z},s)$ is also a defining function for $M$ near $0$.
	This finishes the proof.
\end{proof}
\begin{Rem}\label{BasisCRVF}
	We note that Proposition \ref{genericcoordinates} can be used to give a special local basis of CR vector fields.
	Indeed, let 
	$M\subseteq\C^N$ be a generic submanifold of codimension $d$ that is given locally near a point $p_0\in M$
	by a defining function $\rho=(\rho_1,\dotsc,\rho_d)$. If we use the coordinates $(z,w)\in\C^{n+d}$ from above
	then we can formally view $\rho$ as a function on the variables $(z,\bar{z},w,\bar{w})$.
	Let $\rho_{z}$, $\rho_{\bar{z}}$, $\rho_{w}$ and $\rho_{\bar{w}}$ the Jacobi matrices of $\rho$
	with respect to $z$, $\bar{z}$, $w$ and $\bar{w}$ respectively. 
	We can assume that $\rho_{w}$ and $\rho_{\bar{w}}$ are invertible in a neighbourhood of $p_0$.
	%we obtain using the defining equations \eqref{definingequation} of $M$ near $p_0$ 
	%that there is a local basis of CR vector fields near $p_0$ consisting of CR vector fields of the form
	According to \cite[\S 1.6]{MR1668103} a local basis of CR vector fields near $p_0$ is given by
	\begin{equation*}
	(L)=\bigl(\partial_{\bar{z}}\bigr)-\transp{\rho_{\bar{z}}}\transp{\rho_{\bar{w}}}^{-1}
	\bigl(\partial_{\bar{w}}\bigr)
	\end{equation*}
	where we have used the following notation
	\begin{align*}
	(L)&=\begin{pmatrix}
	L_1\\
	\vdots\\
	L_n
	\end{pmatrix},
	&
	\bigl(\partial_{\bar{z}}\bigr)&=\begin{pmatrix}
	\partial_{\bar{z}_1}\\
	\vdots\\
	\partial_{\bar{z}_n}
	\end{pmatrix},
	&
	\bigl(\partial_{\bar{w}}\bigr)&=\begin{pmatrix}
	\partial_{\bar{w}_1}\\
	\vdots\\
	\partial_{\bar{w}_d}
	\end{pmatrix}.
	\end{align*}
	% in the coordinates $(z,w)$ from Proposition \ref{genericcoordinates} (c.f.\ ).
	If we use the defining function $\rho = t-\varphi$ induced by \eqref{definingequation} then this local basis is of the following form
	\begin{equation*}
	\begin{split}
	L_j&=\frac{\partial}{\partial \bar{z}_j}-\sum_{\mu =1}^d 2b^j_\mu\frac{\partial}{\partial \bar{w}_\mu} \\
	&=\frac{\partial}{\partial \bar{z}_j}-\sum_{\mu =1}^d b^j_\mu\frac{\partial}{\partial s_\mu}
	\end{split}
	\end{equation*}
	with 
	\begin{equation*}
	b^j_\mu=i\frac{\det B_{\mu}^j}{\det \Phi}.
	\end{equation*}
	Here we used
	\begin{equation*}
	\Phi =\rho_{\bar{w}}=\begin{pmatrix}
	1+i(\varphi_1)_{s_1} & \cdots & i(\varphi_1)_{s_d}\\
	\vdots &\ddots &\vdots\\
	i(\varphi_d)_{s_1}&\cdots&1+i(\varphi_d)_{s_d}
	\end{pmatrix}
	\end{equation*}
	and $B^j_\mu$ is the following matrix.
	Let $\delta_{\mu\nu}$ be the Kronecker delta defined by $\delta_{\nu\nu}=1$ and 
	$\delta_{\mu\nu}=0$ otherwise and set
	\begin{align*}
	(\varphi)_{s_\nu}&=\begin{pmatrix}
	\delta_{1\nu}+i(\varphi_1)_{s_\nu}\\
	\vdots\hspace{3ex}\\
	\delta_{d\nu}+i(\varphi_d)_{s_\nu}
	\end{pmatrix} 
	&\text{and} &&
	(\varphi)_{\bar{z}_j}&=\begin{pmatrix}
	(\varphi_1)_{\bar{z}_j}\\
	\vdots\\
	(\varphi_d)_{\bar{z}_j}
	\end{pmatrix}.
	\end{align*}
	Then
	\begin{equation*}
	B_{j\mu}=\begin{pmatrix}
	(\varphi)_{s_1} & \cdots & (\varphi)_{s_{\mu -1}} & %{\color{orange}
	(\varphi)_{\bar{z}_j}%}
	& (\varphi)_{s_{\mu +1}}
	&\cdots
	& (\varphi)_{s_d}
	\end{pmatrix}.
	\end{equation*}
	
	In particular, if $M\subseteq\C^{n+1}$ is a real hypersurface of class $\{\M\}$ locally given by the equation
	$\imag w=\varphi(z,\bar{z},\real w)$ where $\varphi\in\E_\M$ then the vector fields
	\begin{equation*}
	L_j=\frac{\partial}{\partial \bar{z}_j}-2i\frac{\varphi_{\bar{z}_j}}{1+i\varphi_s}\frac{\partial}{\partial \bar{w}}\qquad j=1,\dotsc,n
	\end{equation*}
	form a local basis of the CR vector fields of $M$.
	When we use the local coordinates $(z,\bar{z},s)$ of $M$ induced by \eqref{definingequation} then this basis
	takes the form
	\begin{equation*}
	L_j=\frac{\partial}{\partial \bar{z}_j}-i\frac{\varphi_{\bar{z}_j}}{1+i\varphi_s}\frac{\partial}{\partial s}\qquad j=1,\dotsc,n.
	\end{equation*}
\end{Rem}
Next we give a first result on the structure of ultradifferentiable CR manifolds.
\begin{Def}\label{CRorbit}
	Let $M\subseteq\C^N$ a CR submanifold. The \emph{CR orbit} $\mathrm{Orb}_p$ of $p\in M$ 
	is the local Sussman orbit of $p$ in $M$ relative to the set of ultradifferentiable sections of $T^cM$.
\end{Def}
Note that  if $p_0\in M$ then by construction $T^c_p\mathrm{Orb}_{p_0}=T^c_pM$ for all $p\in\mathrm{Orb}_{p_0}$ thence $\mathrm{Orb}_{p_0}$ is the germ of a CR submanifold of $\C^N$ of CR dimension $n$. 

\begin{Def} Let $M\subseteq\C^N$ a CR manifold and $p_0\in M$. 
	\begin{enumerate}
		\item We say that $M$ is minimal at $p_0$ iff there is no submanifold $S\subseteq M$ through $p_0$ such that 
		$T^c_p M\subseteq T_p S$ for all $p\in S$ and $\dim_\R S <\dim_\R M$.
		\item The manifold $M$ is said to be of finite type at $p_0$ iff there are vector fields
		$X_1,\dotsc,X_k\in \E_\M (M, T^cM)$ such that the Lie algebra generated by the $X_1,\dotsc,X_k$
		evaluated at $p_0$ is isomorphic to $T_{p_0}M$.
	\end{enumerate}
\end{Def}
It is well known that finite type implies minimality and that the two notions coincide for real-analytic CR manifolds, c.f.\ \cite{MR1668103}. 
We are going to show that this fact holds also for quasianalytic CR submanifolds.

\begin{Thm}
	Let $\M$ be a quasianalytic weight sequence and $M\subseteq\C^N$ an ultradifferentiable CR manifold of class $\{\M\}$. The following statements are equivalent:
	\begin{enumerate}
		\item\label{1e} $M$ is minimal at $p_0$.
		\item\label{2e} $\dim_\R \mathrm{Orb}_{p_0}=\dim_\R M$
		\item $M$ is of finite type at $p_0$.
	\end{enumerate}
\end{Thm}
\begin{proof}
	The equivalence of (1) and (2) holds even if $\M$ is non-quasianalytic. Following the arguments in \cite[\S 4.1.]{MR1668103} 
	we see that, if we assume that $M$ is nonminimal then $\dim_\R\mathrm{Orb}_{p_0}<\dim_\R M$.
	On the other hand if $\dim_\R\mathrm{Orb}_{p_0}<\dim_\R M$ then any representative $W$ of $\mathrm{Orb}_{p_0}$ is by the remark below Definition \ref{CRorbit} a submanifold of $M$ and $T^c_p W=T^c_pM$ for all $p\in W$.
	It remains to prove that (2) implies (3).
	
	By Corollary \ref{NaganoSussman} we have that $\mathrm{Orb}_{p_0}=\gamma_{p_0}(\mathfrak{g})$, where $\mathfrak{g}$
	is the Lie algebra generated by the ultradifferentiable sections of $T^cU$ with $U$ being a sufficiently small neighbourhood of $p_0$ and $\gamma_{p_0}(\mathfrak{g})$ the local Nagano leaf of $\mathfrak{g}$ at $p_0$. 
	Hence $\dim_\R \mathrm{Orb}_{p_0}=\dim_\R \gamma_{p_0}(\mathfrak{g})=\dim_\R \mathfrak{g}(p_0)$.
	
	On the other hand $M$ is of finite type at $p_0$ if and only if $\dim_\R \mathfrak{g}(p_0) =\dim_\R M$.
\end{proof}
We shall note we could have shown the equivalence of \eqref{1e} and \eqref{2e} by citing the corresponding  proof in the smooth category in \cite[Theorem 4.1.3.]{MR1668103}. 
Indeed, let $M\subseteq \C^N$ be an ultradifferentiable CR submanifold of class $\{\M\}$ and $p_0\in M$.
Then we can consider $M$ also as an smooth CR manifold and define similar to \cite{MR1668103} 
$\widetilde{\mathrm{Orb}}_{p_0}$ as the Sussman Orbit relative to the smooth sections of $T^cM$ near $p_0$.

However, if $X_1,\dotsc, X_n$ is a local basis of $\E_\M(M,T^cM)$ near $p_0$ then we have that $\mathrm{Orb}_{p_0}$ is generated by $\mathfrak{D}=\{X_1,\dotsc, X_n\}$, c.f.\ Theorem \ref{LocalSussman}.
On the other hand, since the vector fields $X_1,\dotsc, X_n$ constitute also a local basis of $\E(M, T^cM)$ near $p_0$ we obtain also that $\widetilde{\mathrm{Orb}}_{p_0}$ is generated by $\mathfrak{D}$.
It follows that $\mathrm{Orb}_{p_0}=\widetilde{\mathrm{Orb}}_{p_0}$ as germs of manifolds at $p_0$.

The next example is a straightforward generalization of \cite[Example 1.5.16.]{MR1668103}.
\begin{Ex}
	Let $\M$ be a non-quasianalytic weight sequence and $\psi\in\E_\M(\R)$ a real valued function such that
	$\psi(y)=0$ for $y\leq 0$ and $\psi(y)>0$ for $y>0$. We define a real hypersurface in $\C^2$ by
	\begin{equation*}
	M=\bigl\{(z,w)\in\C^2\mid \imag w=\varphi(\imag z)\bigr\}.
	\end{equation*}
	Then $M$ is minimal at the origin but not of finite type at $0$.
	Indeed, if $M$ is non-minimal at $0$ then according to  \cite[Theorem 1.5.15]{MR1668103} 
	there is a holomorphic hypersurface 
	$S\subseteq M$ through the origin. Since $\partial /\partial z$ is tangent to $S$ at $0$ it follows that 
	$S$ is given near the origin by the defining equation $w=h(z)$ where $h$ is a holomorphic function defined
	in some neighbourhood of $0\in\C$ with $h(0)=0$. We conclude that due to $S\subseteq M$ we necessarily have that
	\begin{equation*}
	h(z)-\overline{h(z)}=2i\psi(\real z)
	\end{equation*}
	in some neighbourhood of $0$. It follows that $\psi$ has to be real-analytic near $0$ which contradicts the definition of $\psi$.
	
	Since $\psi$ is flat at the origin, it follows that $M$ cannot be of finite type at $0$.
\end{Ex}

We close this section by recalling the space of multipliers for an ultradifferentiable
abstract CR manifold $(M,\crb)$, which was introduced by \cite{MR3593674} in
the smooth setting.
To begin with consider the following sequence of spaces of sections
\begin{equation*}
E_k=\bigl\langle \mathcal{L}_{K_1}\dots \mathcal{L}_{K_j}\theta\colon j\leq k,\: \: 
K_q\in \E_\M (M, \crb), \,  \theta \in \E_\M (M,T^0 M)
\bigr\rangle.
\end{equation*}
We note that $E_0  = \E_\M (M, T^0 M)$,
and $E_j \subseteq \E_\M(M, T' M) $ for all $j\in \N_0$, and 
set $E = \bigcup_{j\in\N_0} E_j$.

We associate to the increasing chain $E_k$ the increasing sequence
of ideals $\mathcal{S}^k \subset \E_\M (M,\C) $, where
\begin{equation*}
\mathcal{S}^k = \bigwedge\nolimits^N E_k =  
\left\{ \det \begin{pmatrix}
V^1 ( \mathfrak{Y}_1 ) & \dots  & V^1 ( \mathfrak{Y}_N ) \\
\vdots & & \vdots \\
V^N ( \mathfrak{Y}_1 ) & \dots  & V^N ( \mathfrak{Y}_N ) \\
\end{pmatrix} \colon V^j \in E_k,\, \mathfrak{Y}_j \in \E_\M(M,(T'M)^*) \right\}. 
\end{equation*}
We set $\mathcal{S}=\mathcal{S}(M)=\bigcup_{k\in\N_0} \mathcal{S}^k$ and call it the space of multipliers of $M$. 
In fact  each $\mathcal{S}^k$ and thus also $\mathcal{S}$ can be considered actually as ideal sheaves, if
we define $E^k(U)$ and $\mathcal{S}^k(U)$ accordingly.

Note that locally one can find smaller sets of 
generators: Let $U\subset M$ be open, and  assume 
that $L_1,\dots,L_n$ is a local basis for  $\Gamma(U,\crb) $,  
that $\theta^1, \dots ,\theta^d$ is a local basis for $\Gamma(U,T^0M)$, 
and that $\omega^1, \dots, \omega^N$ is a local basis of 
$T^\prime M$. We write 
$\mathcal{L}_j = \mathcal{L}_{L_j}$ for $j = 1, \dots, n$ and 
$\mathcal{L}^\alpha = \mathcal{L}_1^{\alpha_1} \dots \mathcal{L}_n^{\alpha_n}$
for any multi-index $\alpha = (\alpha_1, \dots, \alpha_n)\in \N^n$. 
We note that, since $\crb$ is formally integrable, the $\mathcal{L}^\alpha$, 
where $|\alpha| = k$, generate {\em all} $k$-th order homogeneous differential
operators in the $\mathcal{L}_j$, and  we thus  have 
\begin{equation*}
E_k\big|_U =\bigl\langle \mathcal{L}^\alpha \theta^\mu \,:\;\; 1 \leq \mu\leq d,\: \: |\alpha| \leq k
\bigr\rangle.
\end{equation*}

We can expand
\begin{equation}\label{e:exptheta} \mathcal{L}^\alpha \theta^\mu = \sum_{\ell=1}^{N} A^{\alpha,\mu}_\ell \omega^\ell \end{equation}
and for any choice
$\underline\alpha = (\alpha^1, \dots, \alpha^N ) $  of multi-indices
$\alpha^1, \dots, \alpha^N \in \N^n $ and  $r = (r_1,\dots , r_N) \in \{1,\dots,d\}^N$  
we define the functions
\begin{equation}\label{equ:basisfunctions}
D(\underline\alpha ,r) = \det 
\begin{pmatrix}
A^{\alpha^1,r_1}_1 & \dots & A^{\alpha^1,r_1}_N \\
\vdots & & \vdots \\
A^{\alpha^N,r_N}_1 & \dots & A^{\alpha^N,r_N}_N \\
\end{pmatrix} .
\end{equation} 
With this notation, we have 
\begin{equation*}
\mathcal{S}^k \big|_U = \left\langle D(\underline\alpha ,r)\colon |\alpha^j|\leq k \right\rangle;
\end{equation*}
we shall denote the stalk of $\mathcal{S}^k$ at $p$ by $\mathcal{S}^k_{p}$.

The space of multipliers of a CR manifold $M$ clearly encodes the nondegeneracy properties of $M$.
We close this section by taking a closer look at the connection of $\mathcal{S}$ with finite nondegeneracy.
We recall from \cite{MR1668103} the definition of finite nondegeneracy for abstract CR manifolds.
\begin{Def}\label{FiniteMf}
	Let $M$ be an abstract CR manifold and 
	\begin{equation}\label{finiteNonDeg}
	E_k(p)=\bigl\langle \mathcal{L}_{K_1}\dots \mathcal{L}_{K_j}\theta(p)\colon j\leq k,\: \: 
	K_q\in \E (M, \crb), \,  \theta \in \E (M,T^0 M)
	\bigr\rangle.
	\end{equation}
	for $p\in M$ and $k\in\N$. Then $M$ is $k_0$-nondegenerate at $p_0\in M$
	iff $E_{k_0-1}\subsetneq E_{k_0}=T^\prime_{p_0}M$.
	We say that $M$ is finite nondegenerate iff $M$ is finite nondegenerate at every point.
\end{Def}
\begin{Rem}\label{Mult-FiniteNonDeg}
	This definition is in fact local, since by \cite[Proposition 11.1.10.]{MR1668103} 
	if $L_1,\dotsc, L_n$ is  a local basis of CR vector fields and $\theta^1,\dotsc \theta^d$ is a local basis of characteristic forms near $p_0$ then $M$ is $k_0$-nondegenerate if and only if
	\begin{equation*}
	T^\prime_{p_0}M=\spanc_\C\bigl\{\mathcal{L}^\alpha\theta^{\mu}(p_0)\colon \lvert\alpha\rvert\leq k_0,\;\mu\in\{1,\dotsc,d\}\bigr\}.
	\end{equation*}
	Hence we may replace $M$ with any open neighbourhood $U\!\subseteq\! M$ of $p_0$ in \eqref{finiteNonDeg}.
	Thus we  observe that a CR submanifold $M$ is $k_0$-nondegenerate at $p_0\in M$ 
	if and only if $\mathcal{S}^{k_0}_{p_0}=(\E_\M)_{p_0}$.
	
	More precisely, let $U\subseteq M$ be an open subset and $q\in U$.
	Then $M$ is $k_0$-nondegenerate at $q$ if and only if there is a multiplier 
	$f\in\mathcal{S}^{k_0}(U)$ that does not
	vanish at $q$, i.e.\ $f(q)\neq 0$.
	
	Indeed, if $f(q)\neq 0$ then obviously $E_{k_0}(q)=T^\prime_q M$. On the other hand, if $g(q)=0$ for all 
	multipliers $g\in\mathcal{S}^{k_0}(U)$ then necessarily $E_{k_0}(q)\neq T^\prime_{q}M$.
\end{Rem}
\section{Ultradifferentiable regularity of CR mappings}\label{sec:ultraRefl}
%The aim of this section is to prove the generalizations of results of Bernhard Lamel and Berhanu-Xiao announced in section \ref{Intro}.
%Lamel proved that a finitely nondegenerate CR mapping that extends holomorphically to a wedge 
%between two generic submanifolds is real analytic if the manifolds are real analytic (\cite{MR1861300}) 
%and smooth if the manifolds are both smooth (\cite{MR2085046}).
%Our main result states that if the two CR manifolds are both ultradifferentiable of class $\{\M\}$
%then $H$ has to be ultradifferentiable of the 
%\emph{same} class $\{\M\}$.
The main goal of this section is to present the proof of Theorem \ref{Reflectionsprinciple1}. 
Furthermore we show also ultradifferentiable versions of further regularity results of 
\cite{MR2085046} and \cite{MR3405870}.
However, first we need to recall the definition of finite nondegeneracy of a CR mapping.
\begin{Def}\label{FiniteMap}
	Let $M$ be an abstract CR manifold and $M^\prime\subseteq\C^{N^\prime}$ a generic submanifold.
	Furthermore let $\rho^\prime =(\rho^\prime_1,\dotsc,\rho^\prime_{d^\prime})$ 
	be a defining function of $M^\prime$ near a point $q_0\in M^\prime$, $L_1,\dotsc, L_n$ 
	a local basis of CR vector fields on $M$ near $p_0\in M$ 
	and $H: M\rightarrow M^\prime$ a $\CC^m$-CR mapping with $H(p_0)=q_0$.
	
	For $0\leq k\leq m$ define an increasing sequence of subspaces $E_k(p_0)\subseteq\C^{N^\prime}$ by
	\begin{equation*}
	E_k(p_0):=\spanc_\C\biggl\{L^{\alpha}\frac{\partial\rho^\prime}{\partial Z^\prime}
	\bigl(H(Z),\overline{H(Z)}\bigr)\vert_{Z=p_0}\;:\; 0\leq\lvert\alpha\rvert\leq k,\,1\leq l\leq d^\prime\biggr\}.
	\end{equation*}
	We say that $H$ is $k_0$-nondegenerate at $p_0$ ($0\leq k_0\leq m$) 
	iff $E_{k_0-1}(p_0)\subsetneq E_{k_0}(p_0)=\C^{N^\prime}$.
\end{Def}
\begin{Rem}
Comparing Definition \ref{FiniteMap} with Definition \ref{FiniteMf} 
we observe that a CR submanifold $M\in\C^N$ is $k_0$-nondegenerate
if and only if $\id:\, M\rightarrow M$ is $k_0$-nondegenerate.
We note also the fact that any CR diffeomorphism between 
two $k_0$-nondegenerate CR submanifolds is $k_0$-nondegenerate.
\end{Rem} 

Finally we  need to recall that
if $\rho$ is a local defining function of $M$, 
$\Gamma\subseteq\R^d$ an open convex cone, $p_0\in M$ and $U\subseteq\C^N$ 
an open neighbourhood of $p_0$, 
then a wedge $\mathcal{W}$ with edge $M$ centered at $p_0$ is an open subset of
the form $\mathcal{W}:=\{Z\in U\mid \rho(Z,\bar{Z})\in\Gamma\}$.
%\begin{Thm}
%	Let $M\subseteq\C^N$ and $M^\prime\subseteq\C^{N^\prime}$ be two generic ultradifferentiable 
%	submanifolds of class $\{\M\}$, 
%	$p_0\in M$, $p^\prime_0\in M^\prime$ and $H\!:(M,p_0)\rightarrow (M^\prime,p_0^\prime)$
%	a $\CC^{k_0}$-CR mapping that is $k_0$-nondegenerate at $p_0$. Suppose furthermore that $H$
%	extends continuously to a holomorphic map in a wedge $\W$ with edge $M$.
%	Then $H$ is ultradifferentiable of class $\{\M\}$ in a neighbourhood of $p_0$.
%\end{Thm}
\begin{proof}[Proof of Theorem \ref{Reflectionsprinciple1}]
	Since the assertion of the theorem is local, we are going to work on a neighbourhood 
	$\Omega\subseteq\C^N$ of $p_0$. If $\Omega$ is small enough then by Proposition 
	\ref{genericcoordinates} there are open neigbourhoods $U\subseteq\C^n$ and $V\subseteq\R^d$ 
	of the origin and a function $\varphi\in\E_\M(U\times V,\R^d)$ with $\varphi(0,0)=0$ and 
	$\nabla\varphi (0,0)=0$ such that 
	\begin{equation*}
	M\cap\Omega=\bigl\{(z,w)\in\Omega\mid \imag w=\varphi(z,\bar{z},\real w)\bigr\}.
	\end{equation*}
	From now we denote $M\cap\Omega$ by $M$. If we choose $U$ and $V$ to be small enough we can 
	consider the diffeomorphism
	\begin{align*}
	\Psi:\; U\times V&\longrightarrow \quad M\\
	(z,s)\;&\longmapsto (z,s+i\varphi(z,\bar{z},s)).
	\end{align*}
	If we shrink the neighbourhoods $U,V$ a little bit (such that 
	$\varphi\in\E_\M(\overline{U\!\times\! V},\R^d)$)
	and assume that w.l.o.g.\ both sets are convex
	we can extend the mapping $\Psi$ $\M$-almost analytically in the $s$-variables
	, i.e.\
	there exists a smooth function $\tilde{\Psi}:\, U\times V\times \R^d\rightarrow \C^N$
	such that $\tilde{\Psi}\vert_{U\times V\times\{0\}}=\Psi$ and for each component 
	$\tilde{\Psi}_k$, $k=1,\dotsc,N$, of $\tilde{\Psi}$ we have
	\begin{equation}\label{Diffeo-M-almost}
	\biggl\lvert\frac{\partial \tilde{\Psi}_k}{\partial \bar{w}^\prime_j}(z,\bar{z},s,t)\biggr\rvert
	\leq Ch_\M(\gamma\lvert t\rvert )\qquad j=1,\dotsc, d,
	\end{equation}
	for some constants $C,\gamma>0$. Here $w^\prime=s+it\in V+i\R^d$. 
	We see that there is some $r>0$ such that $\tilde{\Psi}\vert_{U\times V\times B_r(0)}$ 
	is a diffeomorphism.
	
	By assumption $H=(H_1,\dotsc,H_{N^{\prime}})$ extends continuously to a holomorphic mapping
	on a wedge $\W$ near $0$.
	If we shrink $\W$ we may assume that $\partial H_j$, $j=1,\dotsc,N^\prime$, is bounded on $\W$.
	By definition 
	\begin{equation*}
	\W =\bigl\{Z\in \Omega_0\mid \rho(Z,\bar{Z})\in \tilde{\Gamma}\bigr\}
	\end{equation*}
	for a neighbourhood $\Omega_0$ of the origin in $\C^N$ and an open acute cone 
	$\tilde{\Gamma}\subseteq\R^d$. 
	If we shrink $U,V$, when necessary, and choose a suitable open and acute cone $\Gamma$, we achieve that
	\begin{equation*}
	\tilde{\Psi}\Bigl(U\times V\times \Gamma_\delta \Bigr)\subseteq \W
	\end{equation*}
	for some $r\geq \delta>0$. Note that $\tilde{\Psi}(U\times V\times\Gamma_\delta)$ is open in $\C^N$. For each $j=1,\dotsc,N^\prime$ set $h_j=H_j\circ\tilde{\Psi}$ and
	$u_j=H_j\circ \Psi$. Since
	\begin{equation*}
	\frac{\partial h_j}{\partial \bar{w}^\prime_k}=\sum_{\ell=1}^N\frac{\partial H_j}{\partial Z_\ell}
	\frac{\partial \tilde{\Psi}_\ell}{\partial \bar{w}^\prime_k} \qquad j=1,\dotsc, N^\prime,\; k=1,\dotsc, d,
	\end{equation*}
	and $\partial H_j$ is bounded,
	each function $h_j$ is $\M$-almost analytic on $U\times V\times\Gamma_\delta$ due to 
	\eqref{Diffeo-M-almost} and extends $u_j\in\CC^{k_0}(U\times V)$. 
	Hence Theorem \ref{Theorem-M-BVWF} implies
	\begin{equation}\label{Mextension-up}
	\WF_\M u_j\subseteq \bigl(U\times V\bigr)\times \bigl(\R^{2n}\times\Gamma^\circ\bigr)\!\setminus\!\{0\}.
	\end{equation}
	
	If $L_j$, $j=1,\dotsc, n$, is a basis of the CR vector fields on $M=M\cap\Omega$, then 
	$\Lambda_j=\Psi^\ast L_j$ defines a CR structure on $U\times V$ and $\Lambda_j u_k=0$ for
	$j=1,\dotsc, n$ and $k=1,\dotsc, N^\prime$.
	
	Let $\rho^\prime$ be a defining function of $M^\prime$ near $p^\prime_0=0\in\C^{N^\prime}$.
	Then there are ultradifferentiable functions $\Phi_{\ell,\alpha}(Z^\prime,\bar{Z}^\prime,W)$ for
	$\lvert\alpha\rvert\leq k_0$, $\ell=1,\dotsc,d^\prime$, defined in a neighbourhood of
	$\{0\}\times\C^{K_0}\subseteq \C^{N^\prime}\times\C^{K_0}$ and polynomial in the last 
	$K_0=N^\prime\cdot\lvert\{\alpha\in\N^n_0\mid\lvert\alpha\rvert\leq k_0\}\rvert$ variables
	such that 
	\begin{equation}
	\Lambda^\alpha \bigl(\rho_{\ell}^\prime \circ u\bigr)(z,\bar{z},s)=\Phi_{\ell,\alpha}
	\Bigl(u(z,\bar{z},s),\bar{u}(z,\bar{z},s),\bigl(\Lambda^\beta\bar{u}(z,\bar{z},s)
	\bigr)_{\lvert\beta\rvert\leq k_0}\Bigr)=0
	\end{equation}
	and 
	\begin{equation*}
	\Lambda^\alpha\rho^\prime_{\ell,Z^\prime}\bigl(u,\bar{u}\bigr)(0,0,0)
	=\Phi_{\ell,\alpha,Z^\prime}\bigl(0,0,(\Lambda^\beta \bar{u}(0,0,0))_{\lvert\beta\rvert\leq k_0}\bigr)
	\end{equation*}
	Since $H$ is $k_0$-nondegenerate there are multi-indices $\alpha^1,\dotsc,\alpha^{N^\prime}$ 
	and $\ell^1,\dotsc,\ell^{N^\prime}\in\{1,\dots,d^\prime\}$ such that if we set
	\begin{equation*}
	\Phi=\bigl(\Phi_{\ell^1,\alpha^1},\dotsc,\Phi_{\ell^{N^\prime},\alpha^{N^\prime}}\bigr)
	\end{equation*}
	the matrix $\Phi_{Z^\prime}$ is invertible. 
	Hence by Theorem \ref{Lamel-implicit} there is a smooth function $\phi=(\phi_1,\dotsc,\phi_{N^\prime})$ 
	defined in a neighbourhood of $(0,(\Lambda^\beta\bar{u}(0,0,0))_{\lvert\beta\rvert})$ in 
	$\C^{N^\prime}\times\C^{K_0}$ such that, if we shrink $U\times V$ accordingly,
	\begin{equation*}
	u_j(z,\bar{z},s)=\phi_j\bigl(u(z,\bar{z},s),\bar{u}(z,\bar{z},s),(\Lambda^\beta \bar{u}(z,\bar{z},s))_{\lvert\beta\rvert\leq k_0}\bigr)\qquad (z,s)\in U\times V,\;\;j=1,\dotsc,N^\prime
	\end{equation*}
	and \eqref{AlmostholomorphicEst} holds. 
	If we further shrink $U\times V$ and $\delta$ and choose $\Gamma^\prime\subset\subset\Gamma$ appropriately we see that
	\begin{equation}
	g_j (z,\bar{z},s,t)=\phi_j\bigl(h(z,\bar{z},s,-t),\bar{h}(z,\bar{z},s,-t),
	(\tilde{h}_{\ell,\beta}(z,\bar{z},s,t)_{\ell\in\{1,\dotsc,N^\prime\};\lvert\beta\rvert\leq k_0}\bigr)
	\end{equation}
	is well defined for $t\in -\Gamma^\prime_\delta$. 
	Here $\tilde{h}_{j,\beta}$ is the $\M$-almost analytic 
	extension of  $\Lambda^\beta \bar{u}_j$ on $U\times V\times (-\Gamma^\prime_\delta)$, 
	which exists due to  \eqref{Mextension-up}, \eqref{elliptic-regEq}, Proposition \ref{WF-MProperties} 
	and Theorem \ref{BV-M-WF}. 
	It is also easy to see that $\bar{h}(z,\bar{z},s,-t)$ is
	$\M$-almost analytic on $U\times V\times(-\Gamma^\prime_\delta)$.
	We have that
	\begin{equation}\label{Reflection}
	\frac{\partial g_j}{\partial \bar{w}^\prime_\ell}=
	\sum_{k=1}^{N^\prime}\frac{\partial \phi_j}{\partial Z^\prime_k}
	\frac{\partial h_k}{\partial w^\prime_\ell} +\sum_{k=1}^{N^{\prime}} 
	\frac{\partial\phi_j}{\partial\bar{Z}^{\prime}}\frac{\partial\bar{h}}{\partial w^\prime_\ell}
	+\sum_{k=1}^{N^\prime}\sum_{\lvert\beta\rvert\leq k_0}\frac{\partial\phi_j}{\partial W_{k,\beta}}\frac{\partial\tilde{h}_{k,\beta}}{\partial w^\prime_\ell}
	\end{equation}
	for $j=1,\dotsc,N^\prime$ and $\ell=1,\dotsc,d$. Note that we can choose $U\times V$
	and $\Gamma^\prime_\delta$ so small that all functions appearing on the 
	right-hand side are uniformly bounded. 
	Hence, since $\partial_{w_\ell^\prime}\bar{h}=\overline{\partial_{\bar{w}^\prime_\ell} h}$, the last two terms on the right hand side of \eqref{Reflection} are $\M$-almost 
	analytic. The estimate \eqref{AlmostholomorphicEst} and the arguments in 
	\cite[Section 3.3]{LamelMir} give that the first sum on the right hand side
	of \eqref{Reflection} is also $\M$-almost analytic.
	We conclude that $g_j$ is an $\M$-almost analytic extension on 
	$U\times V\times(-\Gamma^\prime_\delta)$ of $u_j$
	and thus
	\begin{equation*}
	\WF_\M u_j\subseteq \bigl(U\times V\bigr)\times\bigl(\R^n\times (\Gamma^\prime\cup -\Gamma^\prime)^\circ\bigr)\!\setminus\!\{0\}
	=\bigl(U\times V\bigr)\times\bigl(\R^n\!\setminus\!\{0\}\times\{0\}\bigr).
	\end{equation*}
	%\begin{center}
	%	\begin{tikzpicture}
	%	\cercle{3,0}{1.5cm}{20}{180}{0.4pt}{}
	%	\draw (2.6,0.7) node{$\mathcal{W}$};
	%	%%\uncover<4->{\cercle{2,0}{0.8cm}{20}{-178}{0.4pt}{blue}}%%
	%	\draw [thick](0,0) to [out=-70,in=120] (6,0);
	%	\draw (5.3,0.9) node{$M$};
	%	\draw[<-] (8.2,2) arc [radius=2.4, start angle=45,  end angle=130];
	%	\node[above] at (6.6,2.7) {$\Psi^{-1}$};
	%	\draw (9,0) rectangle (11,1.3);
	%	\draw[,dashed,thin] (9,0) rectangle (11,-1);
	%	\draw[thick] (7.5,0) -- (12.5,0);
	%	\draw (12.1,-0.25) node{$U\times W$};
	%	\draw (8.5,-0.5) node{$-\Gamma^\prime$};
	%	\draw (8.7,0.7) node{$\Gamma^\prime$};
	%	\end{tikzpicture}
	%\end{center}
	
	On the other hand, since each $u_j$ is CR we have that 
	$\WF_\M u_j\vert_0\subseteq \{0\}\times\R^d\setminus\{0\}$
	by \eqref{elliptic-regEq} and we deduce that
	in fact $\WF_\M u_j\vert_0=\emptyset$ for all $j=1,\dotsc,N^\prime$.
	Hence the mapping $H$ is ultradifferentiable of class $\{\M\}$ near $p_0$.
\end{proof}
%\begin{Rem}
%It is not surprising that the proof of Theorem \ref{Reflectionsprinciple1} uses microlocal techniques.
%In fact, Baouendi, Chang and Treves \cite{MR723811} introduced the socalled hypoanalytic wavefront set
% $\WF_{ha} u\subset T^0M$
%for CR distributions $u$. A point $(p,\xi_p)\in T^0M$ is said to not be in $\WF_{ha} u$ iff there is a wedge %$\mathcal{W}$ with edge $M$
%\end{Rem}
If we recall the well-known result of Tumanov \cite{MR945904} which states that any CR function on a minimal
CR submanifold $M$ extends to a holomorphic function on a wedge with edge $M$, then we obtain the following
corollary.
\begin{Cor}
	Let $M\subseteq\C^N$ and $M^\prime\subseteq\C^{N^\prime}$ generic submanifolds of class $\{\M\}$, 
	$p_0\in M$, $p_0^\prime\in M^\prime$, $M$ minimal at $p_0$ 
	and $H:\,(M,p_0)\rightarrow (M^\prime,p_0^\prime)$
	a  $\CC^{k_0}$-CR mapping that is $k_0$-nondegenerate at $p_0$.
	Then $H$ is ultradifferentiable of class $\{\M\}$ in some neighbourhood of $p_0$.
\end{Cor}

This leads to the following result.
\begin{Cor}
	Let $M\subseteq \C^N$ and $M^\prime\subseteq\C^{N^\prime}$ generic submanifolds of class $\{\M\}$ 
	that are $k_0$-nondegenerate at $p_0\in M$ and $p_0^\prime\in M^\prime$, respectively.
	Furthermore assume that $M$ is minimal at $p_0$ and let $H: M\rightarrow M^\prime$ a CR diffeomorphism
	that is $\CC^{k_0}$ near $p_0$ and satisfies $H(p_0)=p_0^\prime$. 
	Then $H$ has to be ultradifferentiable of class $\{\M\}$ near $p_0$.
\end{Cor}
Recently Berhanu-Xiao \cite{MR3405870} showed that it is possible to slightly weaken the prerequisites 
of the smooth reflection principle of Lamel. 
In particular, the source manifold $M$ can be chosen to be an abstract CR manifold.
Using the methods developed previously we can also generalize this result to the ultradifferentiable category.
\begin{Thm}
	Let $(M,\crb)$ be an abstract CR manifold and $M^\prime\subseteq\C^{N^\prime}$ be a generic submanifold,
	both of class $\{\M\}$.
	Furthermore let $p_0\in M$, $H:\,M\rightarrow M^\prime$ a $\CC^{k_0}$-CR mapping that is 
	$k_0$-nondegenerate at $p_0$ and there is a closed acute cone $\Gamma\subseteq\R^d$ such that
	$\WF_\M H\vert_{p_0}\subseteq \{0\}\times\Gamma$.
	Then $H$ is ultradifferentiable of class $\{\M\}$ near $p_0$.
\end{Thm}
\begin{proof}
	Since the assertation is local we will work on a small chart neighbourhood 
	$\Omega=U\times V\times W\subseteq\R^n\times\R^n\times\R^d$ of $M$ of $p_0=0$. 
	Here $n$ denotes the CR-dimension of $M$ whereas $d$ is the CR-codimension of $M$. 
	We use coordinates $(x,y,s)$ on $\Omega$ and write $z=x+iy$.
	In these coordinates a local  basis of the CR vector fields of $M$ is given by
	\begin{equation*}
	L_j=\frac{\partial}{\partial \bar{z}_j}+\sum_{k=1}^n a_{jk}\frac{\partial}{\partial z_k}
	+\sum_{\alpha=1}^db_{j\alpha}\frac{\partial}{\partial s_\alpha}\qquad j=1,\dotsc,n.
	\end{equation*}
	From the assumptions we conclude that if $\Omega$ is small enough that there is an open, convex cone 
	$\Gamma_1\subseteq\R^{N}\!\setminus\!\{0\}$ such that 
	\begin{equation}
	\WF_\M H=\bigcup_{j=1}^{N^\prime}\WF_\M H_j\subseteq \Omega\times \Gamma_1^\circ
	\end{equation}
	due to the closedness of $\WF_\M H$ in $T^{\ast}M\!\setminus\!\{0\}$. 
	If we further shrink $\Omega$ (resp.\ $U$, $V$ and $W$) and choose an open convex cone 
	$\Gamma_2\subseteq\R^{N}\!\setminus\!\{0\}$  
	such that $\overline{\Gamma}_2\subseteq\Gamma_1\cup\{0\}$ 
	we have by Theorem \ref{BV-M-WF} that there is an $\M$-almost extension $\tilde{F}$ 
	with slow growth of $H$ onto  $\Omega\times\Gamma_2$. If we now choose an open convex cone
	$\Gamma_3\subseteq\R^d\!\setminus\!\{0\}$ with $\{0\}\times\Gamma_3\subseteq\Gamma_2$ we infer that
	\begin{equation*} 
	F:=\tilde{F}\vert_{\Omega\times(\{0\}\times\Gamma_3)}
	\end{equation*}
	is an $\M$-almost analytic function on $U\times V\times W\times \Gamma_3$ with values in $\C^{N^\prime}$ and
	\begin{equation*}
	\lim_{\Gamma_3\ni t\rightarrow 0} F(\,.\,,\,.\,,\,.\,,t)=H
	\end{equation*}
	in the sense of distributions.
	
	Let $\rho^\prime=(\rho_1^\prime,\dotsc,\rho_{N^\prime}^\prime)$ 
	be an ultradifferentiable defining function of $M^\prime$ near $p^\prime_0=H(p_0)$.
	As before in the proof of Theorem \ref{Reflectionsprinciple1} we conclude that there are 
	ultradifferentiable functions $\Phi_{\ell,\alpha}(Z^\prime,\bar{Z}^\prime, W)$ for $\lvert\alpha\rvert\leq k_0$,
	$\ell=1,\dotsc,d^\prime$, defined in a neighbourhood of
	$\{0\}\times\C^{K_0}\subset\C^{N^\prime}\times\C^{K_0}$ and polynomial in the last 
	$K_0=N^\prime\lvert \{\alpha\in\N^{n^\prime}_0\mid \lvert\alpha\rvert\leq k_0\}\rvert$ variables.
	From now on we can follow the proof of Theorem \ref{Reflectionsprinciple1} verbatim.
\end{proof}
\section{Ultradifferentiable regularity of infinitesimal CR automorphisms}
%\section{Infinitesimal CR automorphisms and multipliers}%\label{infinCRIntro}
In this  section we show how the results in \cite{MR3593674} 
concerning the smoothness of infinitesimal CR automorphisms transfer to the ultradifferentiable setting.
Since our presentation here differs in some details from that given in
\cite{MR3593674} we first recall the framework we are going to work in.
In this section $(M,\crb)$ is always an ultradifferentiable abstract CR manifold
of class $\{\M\}$.
\begin{Def}
	Let $U\subseteq M$ an open subset and $X: U\rightarrow TM$ a vector field of class $\CC^1$.
	We say that $X$ is an infinitesimal CR automorphism iff its flow $H^\tau$, defined for small $\tau$, has 
	the property, that there is $\eps >0$ such that $H^\tau$ is a CR mapping provided that $\lvert\tau\rvert\leq \eps$.
\end{Def}
We need for the proofs of the regularity results a more suitable characterization of infinitesimal CR automorphisms.
We call a section $\mathfrak{Y}\in\Gamma(M,(T^\prime M)^\ast)$ a holomorphic vector field on $M$.

Apparently every vector field $X\in\Gamma(M,TM)$ gives rise to a holomorphic vector field by first extending $X$ 
to $\C TM$ and then restricting the extension to $T^\ast M$.
For a partial converse, we recall from \cite{MR3593674} the following purely algebraic result.
\begin{Lem}
	Let $\mathfrak{Y}\in\Gamma(M,(T^\prime M)^\ast)$. 
	Then there exists a unique vector field $X\!\in\!\Gamma(M,TM)$
	such that $\mathfrak{Y}$ is induced by $X$ if and only if $\mathfrak{Y}(\tau)=\overline{\mathfrak{Y}(\tau)}$
	for all characteristic forms $\tau$.
\end{Lem}
%Indeed, since $(\C TM)^{\ast}=\crb^\perp +\bcrb^\perp$ and $\C T^0 M=(\crb \oplus\bcrb)^\perp$, we can decompose any form
%$\omega = \alpha + \bar \beta$ with $\alpha,\beta $ holomorphic forms in a nonunique manner. 
%Thus $\mathfrak{Y}$ gives rise to a real vector field $X $ %= \real \mathfrak{Y}$ 
%via 
%\begin{equation*}
%X(\omega) =%\real \mathfrak{Y}(\omega)=
%\frac{1}{2}\Bigl(\alpha \bigl(\mathfrak{Y}\bigr)+\overline{\beta \bigl(\mathfrak{Y}\bigr)}\Bigr)
%\end{equation*}
%which is well defined provided that $\mathfrak{Y}(\bar{\tau})=\overline{\mathfrak{Y}(\tau)}$ for all $\tau\in\Gamma(M,\C T^0 M)$ or equivalently, that
%$\mathfrak{Y} (\tau)=\overline{\mathfrak{Y}(\tau)}$ for all $\tau\in  \Gamma(M,T^0 M)$, both of which are equivalent 
%to the definition of $X$ above being independent of the decomposition %$\omega=\alpha+\bar{\beta}$. 
%Note that any real vector field $X$ defines by restriction an element of $\Gamma(M,(T'M)^\ast)$ (whose real part it is); 
%We shall not distinguish between $X$ as a real vector field and as an element of $\Gamma(M,(T'M)^\ast)$.
From now on we shall not distinguish between $X$ being a real vector field or a holomorphic vector field.

We recall the well-known identity, see e.g.\ \cite{MR1834454},
\begin{equation*}
\mathcal{L}_X\alpha(Y)=d\alpha (X,Y)+Y\alpha (X)=X\alpha(Y)-\alpha([X,Y]),
\end{equation*}
which holds for arbitrary complex vector fields $X,Y$ and complex forms $\alpha$ on smooth manifolds.

We conclude that accordingly the Lie derivative 
\begin{equation*}
\mathcal{L}_L\omega(\,.\,)=d\omega(L,\,.\,)
\end{equation*}
of a holomorphic form $\omega$ with respect to a CR vector field $L$ is again a holomorphic form.
It is now possible to make the following definition. We shall say that a holomorphic vector field
$\mathfrak{Y}\in\Gamma(M,(T^\prime M)^\ast)$ is CR iff
\begin{equation*}
L\omega(\mathfrak{Y})=d\omega(L,\mathfrak{Y})
\end{equation*}
for every CR vector field $L$ and holomorphic form $\omega$. 
In particular a real vector field $X$ is CR if and only if
\begin{equation*}
\omega([L,X])=0
\end{equation*}
for all CR vector fields $L$ and holomorphic forms $\omega$.
We recall from \cite{MR3593674} the following fact.
\begin{Prop}
	If $X$ is an infinitesimal CR automorphism on $M$, 
	then $X$ considered as a holomorphic vector field, i.e.\ $X\in\CC^1(M,(T^\prime M)^\ast)$ is CR.
\end{Prop}

We are now able to generalize the notion of infinitesimal CR automorphism. 
To this end consider the space $\D^\prime(M,(T^\prime M)^\ast)$ of distributions with values in $(T^\prime M)^\ast$.
\begin{Def}
	An infinitesimal CR diffeomorphism with distributional coefficients on $M$
	is a generalized holomorphic vector field
	$\mathfrak{Y}\in\D^\prime(M,(T^\prime M)^\ast)$ that satisfies
	\begin{equation}\label{CRinfeq1}
	L\omega(\mathfrak{Y})=(\mathcal{L}_L\omega)(\mathfrak{Y})
	\end{equation}
	for every CR vector field $L$ and holomorphic form $\omega$ and
	\begin{equation}\label{InfRefl}
	\mathfrak{Y}(\tau)=\overline{\mathfrak{Y}(\tau)}
	\end{equation}
	for all characteristic forms $\tau$.
\end{Def}
Note that \eqref{CRinfeq1} is in fact a CR equation for $\mathfrak{Y}$.
If $U\subseteq M$ is an open subset of $M$ then we say that $\mathfrak{Y}\in\D^\prime (M,(T^\prime M)^\ast)$ 
is an infinitesimal CR automorphism on $U$ iff \eqref{CRinfeq1} and \eqref{InfRefl} hold for all local sections
$L\in\E_\M(U,\crb\vert_U)$ and $\theta\in\E_\M(U,T^0M\vert_U)$, respectively.
Let the subset $U\subset M$ is small enough such that there is a local basis 
$L_1,\dotsc, L_n$ of CR vector fields and also a local basis
$\{\omega^1,\dotsc,\omega^N\}$ of the space of holomorphic forms.
We recall that locally a distribution $\mathfrak{Y}\in\D^\prime(M,(T^\prime M)^\ast)$ is of the form
\begin{equation}\label{2LokalRep}
\mathfrak{Y}\vert_U=\sum_{j=1}^N X_j\omega_j
\end{equation}
with $X_j\in\D^\prime(U)$. 
%Note that we can expand each $\theta_j =\sum_{\ell=1}^NA_\ell^j\omega^\ell$.
We introduce also the following operators on $U$
\begin{equation*}
\mathbf{L}_j=L_j\cdot\mathbf{Id}_N
=\begin{pmatrix}
L_j & &0\\
&\ddots \\
0 & & L_j
\end{pmatrix}
\end{equation*}
and note that since $d\omega^k (L_j,\,.\,)$ is again a holomorphic form we have
\begin{equation*}
d\omega^k(L_j,\,.\,)=\sum_{\ell =1}^N B_{k,\ell}^j \omega^\ell
\end{equation*}
with $B_{j,\ell}^k\in\E_\M (U)$.
We observe that $\mathfrak{Y}$ is CR on $U$ if and only if
\begin{equation*}
L_j X_k=L_j\bigl(\omega^k(\mathfrak{Y})\bigr)=d\omega^k\bigl(L_j,\mathfrak{Y})\bigr)=\sum_{\ell=1}^NB_{k,\ell}^j X_\ell
\end{equation*}
for all $1\leq j\leq n$ and $0\leq k\leq N$.
We set 
\begin{equation*}
B_j=
\begin{pmatrix}
B_{j,1}^1&\dots &B_{j,N}^1\\
\vdots & & \vdots\\
B_{j,1}^N &\dots &B_{j,N}^N
\end{pmatrix}.
\end{equation*}

Furthermore, using its local representation \eqref{2LokalRep}, we can identify 
$\mathfrak{Y}$ with the vector 
$X=(X_1,\dots,X_N)$. 
Hence \eqref{CRinfeq1} turns into
\begin{align*}
\mathbf{L}_j X&=B_j\cdot X%\label{CRinfeq2}\\
\\
\intertext{or}
P_j X&=0\\
\intertext{respectively, where}
P_j&= \mathbf{L}_j - B_j
\end{align*}
In particular we infer from above and Theorem \ref{elliptic-regThm} that
\begin{equation}\label{InfAutWF}
\WF_\M \mathfrak{Y}\subseteq T^0M.
\end{equation}

%For the formulation of the main regularity results we need one more definition.
%To begin we introduce for the ultradifferentiable CR manifold $M$ the following sequence of spaces of sections.

%\section{Regularity of infinitesimal CR automorphisms}
\begin{Def}
	Let $(M,\crb)$ be an ultradifferentiable abstract CR manifold of class $\{\M\}$, 
	and $\mathfrak{Y}$ an infinitesimal CR diffeomorphism with distributional coefficients of $M$. 
	%see section \ref{infinCRIntro}.
	
	We say  that $\mathfrak{Y}$ extends microlocally to a wedge with edge $M$
	iff there exists a set $\Gamma \subseteq T^0 M$ such that for each $p\in M$, 
	the fiber $\Gamma_p \subseteq T^0_p M\!\setminus \!\{0\}$ is a closed, convex cone, 
	and 
	\begin{equation*}
	\WF_{\M} (\omega(\mathfrak{Y})) \subseteq \Gamma
	\end{equation*}
	for every  holomorphic form $\omega \in \E_{\M}(M,T'M)$.
\end{Def}
Note that the condition $\Gamma\subseteq T^0M$ is not as strict as it seems, because $\WF_\M(\omega(\mathfrak{Y}))\subseteq T^0M$ by \eqref{InfAutWF}.
\begin{Thm}
	\label{thm:multipliers} Let $(M,\crb)$ be an ultradifferentiable abstract  CR structure of class $\{\M\}$, 
	and $\mathfrak{Y}$ an infinitesimal CR diffeomorphism of $M$ with 
	distributional coefficients 
	which extends microlocally to a wedge with edge $M$. 
	
	Then, for any $\omega\in E$, the evaluation $\omega(\mathfrak{Y})$ is ultradifferentiable, 
	and for any $\lambda\in\mathcal{S}$, the vector field $\lambda \mathfrak{Y}$ is also of class $\{\M\}$.
\end{Thm}
\begin{proof}
	Since the assertion is local we will work in a suitable small open set $U\subseteq M$ such that there are 
	local bases $L_1,\dotsc,L_n$ of $\E_\M(U,\crb)$ and $\omega^1,\dotsc,\omega^N$ of $\E_\M(U,T^\prime M)$,
	respectively.
	We recall that we can represent $\mathfrak{Y}$ on $U$ by \eqref{2LokalRep} or 
	by $X=(X_1,\dotsc,X_N)\in\D^\prime(U,\C^N)$.
	By assumption we know that there is a closed convex cone $\Gamma\subseteq T^0M\!\setminus\!\{0\}$ 
	such that $\WF_\M X_j\subseteq \Gamma$ for each $j=1,\dotsc,N$. 
	If we set $W^+=(\Gamma)^c\subseteq T^0M\!\setminus\!\{0\}$, then $\WF_\M X_j\cap W^+=\emptyset$
	for all $j=1,\dotsc,N$. We may refer to this fact by saying that $X_j$ \emph{extends above}.
	On the other hand, if we analogously put $W^-=(-\Gamma)^c\subseteq T^0M\!\setminus\!\{0\}$ then
	$\WF_\M\bar{X}_j\cap W^-=\emptyset$ by \eqref{microreflprop}; 
	we say that $\bar{X}_j$ \emph{extends below}.
	
	Furthermore let $\{\theta^1,\dotsc, \theta^d\}$ be a generating set of $\E_\M(U,T^0M)$ and
	recall \eqref{e:exptheta}, i.e.
	\begin{equation*}
	\mathcal{L}^\alpha\theta^\nu = \sum_{\ell=1}^N A^{\alpha, \nu}_\ell \omega^\ell
	\end{equation*}
	with $A^{\alpha,\nu}_\ell\in\E_\M(U)$ for $\alpha\in\N^n_0$ and $\nu=1,\dotsc, d$.
	In particular, $\eqref{InfRefl}$, i.e.\ $\theta (\mathfrak{Y})=\overline{\theta(\mathfrak{Y})}$, turns into
	\begin{align*}
	\sum_{\ell =1}^N A^{0,\nu}_\ell X_\ell &=\sum_{\ell =1}^N \bar{A}_\ell^{0,\nu}\bar{X}_\ell\\
	\intertext{and applying $\mathcal{L}^\alpha$ to \eqref{InfRefl} yields}
	\sum_{\ell =1}^N A^{\alpha,\nu}_\ell X_\ell &=\sum_{\ell =1}^N
	\sum_{\lvert\alpha\rvert\leq\lvert\alpha\rvert} C^{\beta ,\nu}_\ell L^\beta\bar{X}_\ell,
	\end{align*}
	where $C^{\beta,\nu}_\ell\in\E_\M(U)$.
	Note that in both equations above the left hand side extends above, while the right hand side extends below.
	
	Now choose any $N$-tuple $\underline{\alpha}=(\alpha^1,\dotsc,\alpha^N)\in\N_0^{Nn}$ of multi-indices
	with $\lvert\alpha\rvert\leq k$ for all $j=1,\dotsc,N$ and $r=(r_1,\dotsc,r_N)\in\{1,\dotsc,d\}^N$.
	Then we have
	\begin{equation*}
	\begin{pmatrix}
	A^{\alpha^1,r_1}_1 & \dots & A^{\alpha^1,r_1}_N \\
	\vdots &\ddots & \vdots\\
	A^{\alpha^N,r_N}_1 & \dots & A^{\alpha^N,r_N}_N 
	\end{pmatrix}
	\begin{pmatrix}
	X_1 \\ \vdots \\ X_N
	\end{pmatrix} = 
	\begin{pmatrix}
	\sum
	C^{\alpha^1, \ell}_\beta L^\beta \bar{X}_\ell \\ 
	\vdots \\
	\sum
	C^{\alpha^N, \ell}_\beta L^\beta \bar{X}_\ell
	\end{pmatrix}.
	\end{equation*}
	If we multiply the equation with the classic adjoint of the matrix
	\begin{equation*}
	\begin{pmatrix}
	A^{\alpha^1,r_1}_1 & \dots & A^{\alpha^1,r_1}_N \\
	\vdots &\ddots & \vdots\\
	A^{\alpha^N,r_N}_1 & \dots & A^{\alpha^N,r_N}_N 
	\end{pmatrix}
	\end{equation*}
	then we obtain 
	\begin{equation*}
	D(\underline{\alpha},r)  X_j = \sum_{\substack{|\beta|\leq k \\ \ell=1,\dots, N}} D^{\underline{\alpha},r}_{\beta,j } L^\beta \bar X_j
	\end{equation*}
	for each $j=1,\dotsc, N$ where the $D^{\underline{\alpha},r}_{\beta,j }$ are ultradifferentiable functions 
	on $U$. 
	It follows that the right hand side of this equation  extends below, 
	whereas the left hand side obviously extends above.
	Hence $\WF_\M D(\underline{\alpha},r)X=\emptyset$.
	We conclude that $\lambda X\in\E_\M (U)$ for any $\lambda\in\mathcal{S}^k(U)$
	since $\mathcal{S}^k(U)$ is generated by the functions $D(\underline{\alpha},r)$.
\end{proof}

The next statement is an obvious corollary of Theorem \ref{thm:multipliers}.
\begin{Cor}
	\label{cor:finnondeg}  
	Let $(M,\crb)$ be  finitely nondegenerate
	and $X$ an %locally integrable 
	infinitesimal CR diffeomorphism of $M$ with 
	distributional coefficients 
	which extends microlocally to a wedge with edge $M$. 
	Then $X$ is ultradifferentiable of class $\{\M\}$.
\end{Cor}

However, the condition that $M$ is actually finitely nondegenerate is
far too restrictive. 
We shall say that $(M,\crb)$ is CR-regular 
if for every $p\in M$ there exists a multiplier $\lambda \in \mathcal{S}$ 
with the property that near $p$, the zero set of $\lambda$ is a finite intersection of
real hypersurfaces in $M$, and such that $\lambda$ does not 
vanish to infinite order at $p$.
Thence we can apply Proposition 
\ref{DivProp} or Corollary \ref{DivCor}, respectively.

\begin{Thm}
	\label{thm:main} Let $(M,\crb)$ be an abstract CR structure, $p\in M$, 
	and assume that $M$ is CR-regular near $p$. 
	Then any locally integrable infinitesimal CR diffeomorphism $X$ of $M$ which 
	extends microlocally to a wedge with edge $M$ is of class $\{\M\}$ near $p$. 
\end{Thm}
%Without boundedness conditions on $X$ this theorem is actually 
%in some sense optimal 
%as we are going to see later on. %(see \autoref{sec:example}).

In general it might be difficult to determine if a certain CR manifold is CR-regular.
In the forthcoming we want to present some instances of CR-regular manifolds.
But first we take a closer look at the Lie derivatives of characteristic forms.

Suppose that $M$ is a CR manifold and near a point $p_0\in M$ there are local coordinates $(x,y,s)$ of $M$ such 
that the vector fields
\begin{equation}
L_j=\frac{\partial}{\partial\bar{z}_j}-\sum_{\tau =1}^d b^j_\tau\frac{\partial}{\partial s_\tau},\quad j=1,\dotsc, n,\;z_j=x_j+y_j,
\end{equation}
where $b^j_\tau\in\E_\M$, are a local basis of CR vector fields near $p_0$.  
In this setting (c.f.\ Remark \ref{BasisCRVF}) the characteristic bundle is spanned by the forms
\begin{equation*}
\theta^\tau= ds_\tau +\sum_{j=1}^n b^j_\tau\,d\bar{z}_j+\sum_{j=1}^n\bar{b}^j_\tau\,dz_j,\quad \tau=1,\dotsc, d.
\end{equation*}
Furthermore, the forms $\theta^\tau$, $\tau =1,\dotsc, d$, and $\omega^j=dz_j$, $j=1,\dotsc, n$, constitute
a local basis of holomorphic forms on $M$ near $p_0$.
We also define the functions
\begin{equation*}
\lambda^{j,k}_\mu:= L_k \bar{b}^j_\mu-\bar{L}_j b^k_\mu
\end{equation*}
for $j,k=1,\dotsc, n$ and $\mu=1,\dotsc,d$.

Consider a general holomorphic form
\begin{equation*}
\eta =\sum_{\mu=1}^{d}\sigma_\mu\theta^\mu
+\sum_{j=1}^{n}\rho_j\omega^j.
\end{equation*}
The Lie derivative of $\eta$ with respect to the CR vector field $L_k$ is
\begin{equation}\label{genLieDeriv}
\mathcal{L}_k\eta=d\eta(L_k,\,.\,)=
\sum_{\mu=1}^d\Biggl(L_k\sigma_\mu
-\sum_{\nu=1}^d\sigma_\nu\bigl(b^k_\nu\bigr)_{s_\mu}\Biggr)\theta^\mu
+\sum_{j=1}^n\Biggl(L_k\rho_j 
+\sum_{\mu=1}^d \sigma_\mu\lambda^{j,k}_\mu\Biggr)\omega^j.
\end{equation}

Let $\alpha\in\N_0^n$ a multi-index of length $\lvert\alpha\rvert=m$.
We introduce the finite sequence $m_j:=\sum_{\ell\leq j}\alpha_{\ell}$, $j=1,\dotsc,n$, and set $m_0:=0$
and associate to $\alpha$ the function 
$p_\alpha:\,\{0,1,\dotsc,m\}\rightarrow \{0,1,\dotsc,n\}$
which is defined by
\begin{equation*}
p_\alpha(\ell)=j \qquad \text{if }\; \ell\in (m_{j-1},m_j]
\end{equation*}
for $\ell=1,\dotsc,m$ and $p_\alpha(0)=0$.
We also associate the following sequences of multi-indices to $\alpha$
\begin{align*}
\alpha (\ell)&:=\sum_{q\leq\ell} e_{p_\alpha(q)} & \ell&=0,1,\dotsc,m,\\
\hat{\alpha}(\ell)&:=\sum_{q>\ell} e_{p(q)},
\end{align*}
where $e_j$ is the $j$-th standard unit vector in $\R^n$.

With this notation and \eqref{genLieDeriv} we can now state what the
Lie derivative of the characteristic form $\theta^\mu$ ($\mu=1,\dotsc,d$) is:
\begin{equation}\label{LieDerivChar}
\mathcal{L}^\alpha\theta^\mu=\sum_{\tau=1}^{d}T^{\alpha,\mu}_{\tau}\theta^\tau
+\sum_{j=1}^n A^{\alpha,\mu}_j\omega^j 
\end{equation}
The functions $T^{\alpha,\mu}_\tau$ and $A^{\alpha,\mu}_j$ are defined iteratively by
\begin{subequations}\label{gencoeff-formula}
	\begin{align}
	T^{0,\mu}_\tau&=\delta_{\mu\tau},\notag\\
	T^{\alpha,\mu}_{\tau} &=L_{p_\alpha(1)}T^{\hat{\alpha}(1),\mu}_\tau
	-\sum_{\nu=1}^{d}\bigl(b^{p(1)}_\nu\bigr)_{s_\tau}T^{\hat{\alpha}(1),\mu}_\nu\\
	\intertext{and}
	A^{\alpha,\mu}_{j} &=\sum_{k=1}^{m}\sum_{\nu=1}^{d}L^{\alpha(k-1)}\Bigl(T^{\alpha -\alpha(k),\mu}_\nu\lambda^{j,p_\alpha(k)}_\nu\Bigr).
	\end{align}
\end{subequations}

We are now able to give the first example of a CR regular submanifold of $\C^N$.
\begin{Def}
	We say that a real hypersurface $M\subseteq\C^N$ is weakly nondegenerate at $p_0$ iff there 
	exist coordinates $(z,w)\in\C^n\times\C$ near $p_0$ and  numbers $k,m\in\N $ 
	such that $p_0=0$ in these coordinates and near $p_0$
	$M$ is given by an equation of the form
	\begin{align*}
	\imag w &= (\real w)^{m} \varphi (z,\bar z, \real w),  \\
	\intertext{where} 
	\frac{\partial^{|\alpha| }\varphi }{\partial z^\alpha  } (0,0,0) &= \frac{\partial^{|\alpha|}\varphi }{\partial \bar z^\alpha} (0,0,0) = 0, \quad |\alpha| \leq k, 
	\end{align*}
	and 
	\[ \spanc_\C \{ \varphi_{z{\bar z}^\alpha}  (0,0,0) \colon |\alpha|\leq k \} = \C^n.\]
	If $k_0$ is the smallest $k$ for which the preceding condition holds, 
	we say that $M$ is weakly $k_0$-nondegenerate at $p_0$. 
\end{Def}

\begin{Prop}
	\label{cor:main2a}  Let $M\subseteq \C^N$ be an %generic
	ultradifferentiable real hypersurface, $p_0\in M$, and 
	assume that  $M$ is weakly $k_0$-nondegenerate  at $p_0$.
	Then $M$ is CR regular near $p_0$.
	In particular, any locally integrable infinitesimal CR diffeomorphism of $M$ 
	which extends microlocally to a wedge with 
	edge $M$ near $p_0$ is ultradifferentiable near $p_0$. 
\end{Prop}
\begin{proof}
	In order to show that $M$ is CR regular
	we are going to construct a multiplier $\lambda\in\mathcal{S}$ of the form
	\begin{equation*}
	\lambda(z,\bar{z},s)=s^\ell\psi(z,\bar{z},s)
	\end{equation*}
	in suitable local coordinates and with $\psi\in\E_\M$ not vanishing at $s=0$ and $\ell\in\N$.
	
	Recall that by assumption there are coordinates $(z,w)\in\C^n\times\C$
	such that $p_0=0$ and $M$ is given locally by
	\begin{equation*}
	\imag w=(\real w)^m\varphi (z,\bar{z},\real w)
	\end{equation*}
	where $m\!\in\!\N$ and $\varphi$ is an ultradifferentiable real-valued function defined near $0$
	with the property that $\varphi_{z^\alpha}(0)\!=\!\varphi_{\bar{z}^\alpha}(0)\!=\!0$ 
	for $\lvert\alpha\rvert\leq k_0$ and
	\begin{equation*}
	\spanc_\C \{ \varphi_{z{\bar z}^\alpha}  (0,0,0) \colon 0<\lvert\alpha\rvert\leq k_0\}%\in\N^n\} 
	= \C^n.
	\end{equation*}
	
	In these coordinates a local basis of the CR vector fields on $M$ is given by
	\begin{align*}
	L_j&=\frac{\partial}{\partial \bar{z}_j}-b^j\frac{\partial}{\partial s},\qquad 1\leq j\leq n,\\
	\intertext{with}
	b^j&=i\frac{s^m\varphi_{\bar{z}_j}}{1+i(s^m\varphi)_s},
	\end{align*}
	whereas the characteristic bundle is spanned near the origin by 
	\begin{equation*}
	\theta =ds+\sum_{j=1}^n b^j\,d\bar{z}_j
	+\sum_{j=1}^n b^j\,dz_j
	\end{equation*}
	and $\theta$ together with the forms $\omega^j=dz_j$ constitute a local basis of $T^\prime M$ near the origin.
	%We claim that $A^{\alpha}_\ell=s^m B^{\alpha}_\ell$ 
	%where the $B^{\alpha}_\ell$ are some ultradifferentiable functions  and 
	%$B^{\alpha}_\nu(0)=2i\varphi_{\bar{z}^{\alpha}z_\nu}(0)$ for 
	%$\lvert\alpha\rvert\leq k$.
	
	We observe that for $1\leq j,\ell\leq n$  
	\begin{align*}
	\lambda^{j}_\ell&:=L_j\bar{b}^\ell-\bar{L}_\ell b^j\\
	&=s^m\Biggl(\frac{i\varphi_{\bar{z}_j z_\ell}
		(1+i(s^m\varphi)_s)+ \varphi_{z_\ell}(s^m\varphi_{\bar{z}_j})_s}{(1+i(s^m\varphi)_s)^2}\\
	&\qquad+\frac{\varphi_{\bar{z}_j}\bigl((s^m\varphi_{z_\ell})_s(1+i(s^m\varphi)_s)
		-is^m\varphi_{z_\ell}(s^m\varphi)_{ss}\bigr)}{(1+i(s^m\varphi)_s)^3}\\
	&\qquad+\frac{i\varphi_{\bar{z}_j z_\ell}
		(1+i(s^m\varphi)_s) + \varphi_{\bar{z}_j}(s^m\varphi_{z_\ell})_s}{(1+i(s^m\varphi)_s)^2}\\
	&\qquad-\frac{\varphi_{z_\ell}\bigl((s^m\varphi_{\bar{z}_j})_s(1+i(s^m\varphi)_s)
		-s^m\varphi_{\bar{z}_j}(s^m\varphi)_{ss}\bigr)}{(1+i(s^m\varphi)_s)^3}\Biggr)\\
	&=s^m\chi^j_\ell
	\end{align*}
	and 
	$\chi^j_\ell(0)=2i\varphi_{\bar{z}_jz_\ell}(0)$
	by the assumptions on $\varphi$. 
	
	In this setting \eqref{LieDerivChar} takes the form
	\begin{equation*}
	\mathcal{L}^\alpha\theta =T^\alpha\theta +\sum_{j=1}^n A^\alpha_j\omega^j
	\end{equation*}
	and \eqref{gencoeff-formula} implies that
	\begin{align*}
	T^\alpha &=L_{p(1)}T^{\hat{\alpha}(1)}-\bigl(b^{p(1)}\bigr)_s T^{\hat{\alpha}(1)}, \quad T^0=1,\\
	A^\alpha_j &=\sum_{k=1}^{\lvert\alpha\rvert}=L^{\alpha(k-1)}\Bigl(T^{\hat{\alpha}(k)}\lambda^j_{p(k)}\Bigr).
	\end{align*}
	
	If we use  the two simple facts for smooth functions $f,g$, namely $(s^qf)_s=s^{q-1}f+s^qf_s$ for $q\in\N$
	we see that $T^\beta=s^{m-1}G^\beta$ for $\lvert\beta\rvert\geq 1$.
	Hence, if $m\geq 2$ we have
	\begin{equation*}
	A^{\alpha}_\ell (z,\bar{z},s)=s^m 
	\frac{2i\varphi_{\bar{z}^\alpha z_\ell}(z,\bar{z},s)}{1+(s^m\varphi(z,\bar{z},s))_s^2}
	+s^{2m-1}R^\alpha_\ell(z,\bar{z},s)=s^m B^{\alpha}_\ell(z,\bar{z},s).
	\end{equation*}
	On the other hand we obtain for $m=1$ the following representation
	\begin{equation*}
	A^{\alpha}_\ell(z,\bar{z},s)
	=s\frac{2i\varphi_{\bar{z}^\alpha z_\ell}(z,\bar{z},s)}{1+(\varphi(z,\bar{z},s)+s\varphi_s(z,\bar{z},s))^2}
	+sS^{\alpha}_\ell(z,\bar{z},s)+s^2R^{\alpha}_\ell(z,\bar{z},s)=sB^{\alpha}_\ell(z,\bar{z},s),
	\end{equation*}
	where $S^{\alpha}_\ell$ is a sum of products of rational functions 
	with respect to $\varphi$ and its derivatives.
	Each of these summands contains at least one factor of the form 
	$\varphi_{\bar{z}^\beta}$ or $\varphi_{z^\beta}$
	with $\lvert\beta\rvert\leq\lvert\alpha\rvert\leq k_0$ 
	and therefore $S^{\alpha}_\ell(0)=0$. 
	%Hence $B^{\alpha}_\nu(0)=2i\varphi_{z\bar{z}^\alpha}(0)$.
	
	By assumption there have to be multi-indices $\alpha^1,\dots,\alpha^n\neq 0$ of length shorter than 
	$k_0$ such that 
	\begin{equation*}
	\{\varphi_{z\bar{z}^{\alpha^1}}(0),\dots ,\varphi_{z\bar{z}^{\alpha^n}}(0)\}
	\end{equation*}
	is a basis for $\C^n$.
	Now we choose $\underline{\alpha}=(0,\alpha^1,\dots,\alpha^n)$ and calculate according to 
	\eqref{equ:basisfunctions} the multiplier $D(\underline{\alpha})=D(\underline{\alpha}, 1)$
	(note that $d=1$):
	\begin{align*}
	D(\underline{\alpha})&=\det \begin{pmatrix}
	1 & 0 & \dots &0\\
	A^{\alpha^1}_\theta & A^{\alpha^1}_1 & \dots & A^{\alpha^1}_n\\
	\vdots & \vdots & \ddots & \vdots\\
	A^{\alpha^n}_\theta & A^{\alpha^n}_1 &\dots & A^{\alpha^n}_n
	\end{pmatrix}\\
	&=s^{n\cdot m}\det\begin{pmatrix}
	1 & 0 & \dots &0\\
	A^{\alpha^1}_\theta & B^{\alpha^1}_1 & \dots & B^{\alpha^1}_n\\
	\vdots & \vdots & \ddots & \vdots\\
	A^{\alpha^n}_\theta & B^{\alpha^n}_1 &\dots & B^{\alpha^n}_n
	\end{pmatrix}\\
	&=s^{n\cdot m}Q(\underline{\alpha})\\
	\intertext{where}
	Q(\underline{\alpha})&=\det\begin{pmatrix}
	1 & 0 & \dots &0\\
	A^{\alpha^1}_\theta & B^{\alpha^1}_1 & \dots & B^{\alpha^1}_n\\
	\vdots & \vdots & \ddots & \vdots\\
	A^{\alpha^n}_\theta & B^{\alpha^n}_1 &\dots & B^{\alpha^n}_n 
	\end{pmatrix}
	=\det\begin{pmatrix}
	B^{\alpha^1}_1 & \dots & B^{\alpha^1}_n\\
	\vdots & \ddots & \vdots\\
	B^{\alpha^n}_1 &\dots & B^{\alpha^n}_n
	\end{pmatrix},\\
	\intertext{hence}
	Q(\underline{\alpha})(0)&=(2i)^n\det \begin{pmatrix}
	\varphi_{z\bar{z}^{\alpha^1}}(0)\\
	\vdots\\
	\varphi_{z\bar{z}^{\alpha^n}}(0)
	\end{pmatrix}
	\neq 0.
	\end{align*}
	We conclude that $M$ is CR-regular.
\end{proof}
Obviously, a similar approach as in the hypersurface case above can be used to find manifolds of higher codimension that are CR-regular. 
%However for the sake of simplicity we restrict ourselves to the case of codimension $2$.
\begin{Def}\label{higherCodimDef}
	We say that a CR manifold $M\subseteq\C^{N}$ of codimension $d$ is  weakly
	nondegenerate at $p_0\in M$ (in the first codimension)
	iff there are local coordinates $(z,w)\in\C^{n+d}$ near $p_0$
	such that $M$ is given by the equations
	\begin{equation*}
	\imag w_\mu = (\real w)^{\gamma^\mu} \varphi_\mu (z,\bar z, \real w),\quad \mu=1,\dotsc,d,
	\end{equation*}
	with $\gamma^1<\gamma^\nu$, $\nu=2,\dotsc,d$, and $\lvert\gamma^1\rvert\geq 2$.
	Furthermore the function $\varphi_1$ satisfies for some $k$
	\begin{equation*}
	\spanc_\C\bigl\{\big(\varphi_1\big)_{z\bar{z}^\alpha}(0,0,0):\;\lvert\alpha\rvert\leq k\bigr\}=\C^n.
	\end{equation*}
	If $k_0$ is the smallest integer $k$ for which the above condition holds, we say that $M$ is weakly $k_0$-nondegenerate at $p_0$.
\end{Def}
\begin{Prop}\label{higherCodimRes}
	Let $M\subseteq \C^{N}$ be a generic ultradifferentiable
	CR submanifold of codimension $d$, $p_0\in M$, and 
	assume that  $M$ is weakly nondegenerate  at $p_0$.
	Then any locally integrable infinitesimal CR diffeomorphism of $M$ which extends microlocally to a wedge with 
	edge $M$ near $p_0$ is ultradifferentiable 
	near $p_0$. 
\end{Prop}
\begin{proof}
	Similar to before we have to construct a multiplier $\lambda\in\mathcal{S}$ of the form 
	$\lambda(z,\bar{z},s)=s^\beta\psi(z,\bar{z},s)$ where $\psi\in\E_\M$ and $\psi(0)\neq 0$.
	By assumption there are coordinates $(z,w)\in\C^{n+d}$ near $p_0=0$ such that $M$ is given by
	\begin{equation*}
	\imag w_\mu = (\real w)^{\gamma^\mu} \varphi_\mu (z,\bar z, \real w),\quad \mu=1,\dotsc,d.
	\end{equation*}
	In particular note that $\alpha^1\leq\alpha^\mu$ for $\mu=2,\dotsc,d$.
	
	We deduce from Remark \ref{BasisCRVF} that the vector fields
	\begin{equation*}
	L_j=\frac{\partial}{\partial \bar{z}_j}-\sum_{\mu=1}^d b^j_\mu\frac{\partial}{\partial s_\mu}
	\end{equation*}
	are a local basis of the CR vector fields near the origin. 
	The coefficients $b^j_\mu$ are of the form
	\begin{equation*}
	b^j_\mu =i\big(\det (\Id_d +i\Phi)\big)^{-1} \cdot \det B^j_\mu
	\end{equation*}
	where $\Phi$ denotes the Jacobi matrix of the map $(s^{\gamma^\mu}\varphi_\mu)_\mu$ with respect to the variables $s=(s_1\dotsc,s_d)$ and
	\begin{equation*}
	B^j_\mu =\begin{pmatrix}
	1+i(s^{\gamma^1}\varphi_1)_{s_1}&\dots & i(s^{\gamma^1}\varphi_1)_{s_{\mu -1}} 
	& s^{\gamma^1}(\varphi_1)_{\bar{z}_j}
	& i(s^{\gamma^1}\varphi_1)_{s_{\mu +1}} &\dots &i(s^{\gamma^1}\varphi_1)_{s_d}\\
	\vdots & &\vdots &\vdots &\vdots & &\vdots\\
	i(s^{\gamma^{\mu}}\varphi_{\mu})_{s_1}& \dots & i(s^{\gamma^{\mu}}\varphi_{\mu})_{s_{\mu -1}}
	&s^{\gamma^{\mu}}(\varphi_{\mu})_{\bar{z}_j} 
	& i(s^{\gamma^{\mu}}\varphi_{\mu})_{s_{\mu +1}} &\dots & i(s^{\gamma^{\mu}}\varphi_{\mu})_{s_d}\\
	\vdots&&\vdots&\vdots&\vdots&&\vdots\\
	i(s^{\gamma^d}\varphi_d)_{s_1}&\dots &i(s^{\gamma^d}\varphi_d)_{s_{\mu -1}}
	&s^{\gamma^d}(\varphi_d)_{\bar{z}_j}
	&i(s^{\gamma^d}\varphi_d)_{s_{\mu +1}} &\dots & 1+i(s^{\gamma^d}\varphi_d)_{s_d}
	\end{pmatrix}.
	\end{equation*}
	Hence for all $j=1,\dotsc n$ and $\mu =1,\dotsc, d$ we have
	\begin{equation}\label{CoeffForm1}
	b^j_\mu = is^{\gamma^1}\big(\det (\Id_d +i\Phi)\big)^{-1}\det C^j_\mu
	\end{equation}
	with
	\begin{equation*}
	C^j_\mu =\begin{pmatrix}
	1+i(s^{\gamma^1}\varphi_1)_{s_1}&\dots & i(s^{\gamma^1}\varphi_1)_{s_{\mu -1}} 
	& (\varphi_1)_{\bar{z}_j}
	& i(s^{\gamma^1}\varphi_1)_{s_{\mu +1}} &\dots &i(s^{\gamma^1}\varphi_1)_{s_d}\\
	\vdots & &\vdots &\vdots &\vdots & &\vdots\\
	i(s^{\gamma^{\mu}}\varphi_{\mu})_{s_1}& \dots & i(s^{\gamma^{\mu}}\varphi_{\mu})_{s_{\mu -1}}
	&s^{\tilde{\gamma}^{\mu}}(\varphi_{\mu})_{\bar{z}_j} 
	& i(s^{\gamma^{\mu}}\varphi_{\mu})_{s_{\mu +1}} &\dots & i(s^{\gamma^{\mu}}\varphi_{\mu})_{s_d}\\
	\vdots&&\vdots&\vdots&\vdots&&\vdots\\
	i(s^{\gamma^d}\varphi_d)_{s_1}&\dots &i(s^{\gamma^d}\varphi_d)_{s_{\mu -1}}
	&s^{\tilde{\gamma}^d}(\varphi_d)_{\bar{z}_j}
	&i(s^{\gamma^d}\varphi_d)_{s_{\mu +1}} &\dots & 1+i(s^{\gamma^d}\varphi_d)_{s_d}
	\end{pmatrix}
	\end{equation*}
	and $\tilde{\gamma}^\mu=\gamma^\mu-\gamma^1>0$.
	We observe that
	\begin{subequations}\label{CoeffForm2}
		\begin{align}
		\det C^j_1\big\vert_{s=0}&=(\varphi_1)_{\bar{z}_j}(z,\bar{z},0) \\ %+ R^j_1\\
		\det C^j_\mu &= 0 &\mu &=2,\dotsc,d,%R^j_\mu, &\mu &=2,\dotsc,d,
		\end{align}
	\end{subequations}
	since $\lvert\gamma^\mu\rvert \geq\lvert\gamma^1\rvert\geq 2$. %and the functions $R^j_\mu$ satisfy $R^j_\mu=O(\lvert s\rvert)$ for $s\rightarrow 0$ and $\mu =1,\dotsc, d$.
	
	Furthermore the forms
	\begin{equation*}
	\theta^\mu =ds_\mu +\sum_{j=1}^n b^j_\mu d\bar{z}_j +\sum_{j=1}^n\bar{b}^j_\mu dz_j,\quad\mu=1,\dotsc,d,
	\end{equation*}
	span the characteristic bundle near $0$ and $\theta^\mu$, $\mu=1,\dotsc,d$
	and $\omega^j=dz_j$, $j=1,\dotsc, n$, form
	a local basis of the holomorphic forms on $M$.
	From \eqref{LieDerivChar} we recall  for $\alpha\in\N_0^n$ and $\mu=1,\dotsc,d$ that
	\begin{equation*}
	\mathcal{L}^\alpha\theta^\mu=\sum_{\tau=1}^{d}T^{\alpha,\mu}_{\tau}\theta^\tau
	+\sum_{j=1}^n A^{\alpha,\mu}_j\omega^j 
	\end{equation*}
	and from \eqref{gencoeff-formula}
	\begin{align*}
	T^{0,\mu}_\tau&=\delta_{\mu\tau}\\
	T^{\alpha,\mu}_{\tau} &=L_{p_\alpha(1)}T^{\hat{\alpha}(1),\mu}_\tau
	-\sum_{\nu=1}^{d}\bigl(b^{p(1)}_\nu\bigr)_{s_\tau}T^{\hat{\alpha}(1),\mu}_\nu\\
	A^{\alpha,\mu}_{j} &=\sum_{k=1}^{\lvert\alpha\rvert}
	\sum_{\nu=1}^{d}L^{\alpha(k-1)}\Bigl(T^{\alpha -\alpha(k),\mu}_\nu\lambda^{j,p_\alpha(k)}_\nu\Bigr).
	\end{align*}
	We recall that
	\begin{align*}
	\lambda^{j,k}_\nu &=L_k\bar{b}^j_\nu -\bar{L}_jb^k_\nu\\
	&=\big(\bar{b}^j_\nu\big)_{\bar{z}_k}-\sum_{\mu=1}^d b^k_\mu \big(\bar{b}^j_\nu\big)_{s_\mu}
	-\big(b^k_\nu\big)_{z_j}+\sum_{\mu=1}\bar{b}^j_\mu \big(b^k_\nu\big)_{s_\mu}
	\end{align*}
	and note that \eqref{CoeffForm1} and \eqref{CoeffForm2} imply that
	\begin{align*}
	\lambda^{j,k}_\nu &= 2is^{\gamma^1}R^{j,k}_\nu & \nu&=1,\dotsc,d,\\
	\intertext{where} 
	R_1^{j,k}\Big\vert_{s=0}&=\big(\varphi_1\big)_{\bar{z}_kz_j}\Big\vert_{s=0}\\
	R_\nu^{j,k}\Big\vert_{s=0}&=0 & \nu&=1,\dotsc,d.
	\end{align*}
	It is easy to see that also $T^{\alpha,\mu}_\tau\big\vert_{s=0} =0$ for $\alpha\neq 0$.
	We conclude that for all $\alpha\neq 0$, and $j=1,\dotsc,n$
	\begin{align*}
	A^{\alpha,\mu}_j &=2is^{\gamma^1}\tilde{A}^{\alpha,\mu}_j & \mu &=1,\dotsc ,d\\
	\intertext{where}
	\tilde{A}^{\alpha,1}_j\Big\vert_{s=0}& =\big(\varphi_1\big)_{\bar{z}^\alpha z_j}\Big\vert_{s=0}\\
	\tilde{A}^{\alpha,\mu}_j\Big\vert_{s=0} &=0 & \mu&=2,\dotsc, d.
	\end{align*}
	
	By assumptation there are multi-indices $\alpha^1,\dotsc,\alpha^n\in\N_0^n$ of length at most $k_0$ such 
	that the vectors
	\begin{equation*}
	\bigl(\varphi_1\bigr)_{z\bar{z}^{\alpha^j}}(0),\qquad j=1,\dotsc,n,
	\end{equation*}
	form a basis of $\C^n$. 
	
	We compute the multiplier $D(\overline{\alpha},r)$ for 
	$\underline{\alpha}=(0,\dotsc,0,\alpha^1,\dotsc,\alpha^n)$
	and $r=(1,2,\dotsc,d,1,\dotsc,1)$. 
	By \eqref{equ:basisfunctions} we have
	\begin{align*}
	D(\underline{\alpha},r)&=\det \begin{pmatrix}
	1&\dots&0&0&\dots&0\\
	\vdots & &\vdots&\vdots &&\vdots\\
	0&\dots &1&0&\dots&0\\
	T^{\alpha^1,1}_1 &\dots &T^{\alpha^1,1}_d& A^{\alpha^1,1}_1&\dots &A^{\alpha^1,1}_n\\
	\vdots&&\vdots&\vdots&&\vdots\\
	T^{\alpha^n,1}_1 &\dots &T^{\alpha^n,1}_d&A^{\alpha^n,1}_1&\dots &A^{\alpha^n,1}_n
	\end{pmatrix}\\
	&=\det\begin{pmatrix}
	A^{\alpha^1,1}_1&\dots &A^{\alpha^1,1}_n\\
	\vdots&&\vdots\\
	A^{\alpha^n,1}_1&\dots &A^{\alpha^n,1}_n
	\end{pmatrix}\\
	&=\det\begin{pmatrix}
	2is^{\gamma^1}\tilde{A}^{\alpha^1,1}_1&\dots &2is^{\gamma^1}\tilde{A}_n^{\alpha^1,1}\\
	\vdots&&\vdots\\
	2is^{\gamma^1}\tilde{A}_1^{\alpha^n,1} &\dots &2is^{\gamma^1}\tilde{A}_n^{\alpha^n,1}
	\end{pmatrix}\\
	&=(2i)^n s^{n\gamma^1}\det\begin{pmatrix}
	\tilde{A}^{\alpha^1,1}_1 &\dots &\tilde{A}^{\alpha^1,1}_n\\
	\vdots&&\vdots\\
	\tilde{A}^{\alpha^n,1}_1 &\dots &\tilde{A}^{\alpha^n,1}_n
	\end{pmatrix}\\
	&=(2i)^n s^{n\gamma^1}\Lambda(\underline{\alpha},r).
	\end{align*}
	We conclude
	\begin{equation*}
	\Lambda(\underline{\alpha},r)(0)=\det
	\begin{pmatrix}
	\big(\varphi_1\big)_{z\bar{z}^{\alpha^1}}(0)\\ \vdots\\ \big(\varphi_1\big)_{z\bar{z}^{\alpha^n}}(0)
	\end{pmatrix}\neq 0.
	\end{equation*}
\end{proof}
In the preceding results we required the involved manifolds to have a special form in order to simplify the necessary calculations, but of course there are many more CR regular manifolds.
The next example gives a CR manifold that is not weakly nondegenerate at $0$ in the sense of Definition 
\ref{higherCodimDef} but is still CR regular.
\begin{Ex}
	Let $M\subseteq\C^3$ be the CR manifold given by
	\begin{align*}
	\imag w_1&= \real w_1\,\lvert z\rvert^2\\
	\imag w_2 &=\real w_2\,\lvert z\rvert^2.
	\end{align*}
	The CR bundle $\crb$ of $M$ is spanned by
	\begin{align*}
	L&=\frac{\partial}{\partial \bar{z}}-i\frac{s_1z}{1+i\lvert z\rvert^2}\frac{\partial}{\partial s_1}-i\frac{s_2 z}{1+i\lvert z\rvert^2}\frac{\partial}{\partial s_2}.\\
	\intertext{Thus a basis of the characteristic form is given by}
	\theta^1&=ds_1+ i\frac{s_1z}{1+i\lvert z\rvert^2}d\bar{z}-i\frac{s_1\bar{z}}{1-i\lvert z\rvert^2}dz\\
	\theta^2&=ds_2+ i\frac{s_2z}{1+i\lvert z\rvert^2}d\bar{z}-i\frac{s_2\bar{z}}{1-i\lvert z\rvert^2}dz.
	\end{align*}
	
	We know that $\theta^1$, $\theta^2$ and $\omega =dz$ form a basis of $T^\prime M$.
	If $\alpha=e_1$ we recall from \eqref{LieDerivChar} that
	\begin{equation*}
	\mathcal{L}^\alpha\theta^1=T^{\alpha,1}_{1}\theta^1 +T^{\alpha,1}_{2}\theta^2
	+ A^{\alpha,1}\omega.
	\end{equation*}
	Using \eqref{gencoeff-formula} we observe that
	\begin{align*}
	T^{\alpha,1}_1&=-i\frac{z}{1+i\lvert z\rvert^2}\\
	T^{\alpha,1}_2&=0\\
	A^{\alpha,1}&=-2is_1\frac{1-\lvert z\rvert^4}{\bigl(1+\lvert z\rvert^4\bigr)^2}.
	\end{align*}
	Hence, if we set $\underline{\alpha}=(0,0,\alpha)$ and $r=(1,2,1)$ 
	then the multiplier $D(\underline{\alpha},r)$ of $M$ given by \eqref{equ:basisfunctions} is
	\begin{equation*}
	D(\underline{\alpha},r)=\det \begin{pmatrix}
	1 & 0& 0\\
	0& 1& 0\\
	-i\tfrac{z}{1+i\lvert z\rvert^2}&0 &-2is_1\tfrac{1-\lvert z\rvert^4}{(1+\lvert z\rvert^4)^2}
	\end{pmatrix}
	=-2is_1\frac{1-\lvert z\rvert^4}{\bigl(1+\lvert z\rvert^4\bigr)^2}
	\end{equation*}
	and thus $M$ is CR regular.
\end{Ex}
We could now give an ultradifferentiable version of the example given in section $7$
of \cite{MR3593674} in order to show that in the previous statements 
the requirement on 
the infinitesimal automorphisms to be locally integrable is essential for the assertations to hold.
However, to do this it would be enough to replace
everywhere in section $7$ of \cite{MR3593674} the word \emph{smooth} with the term 
\emph{ultradifferentiable of class $\{\M\}$}.

Instead we take a closer look into the case of quasianalytic manifolds.
We begin with recalling the following definition from \cite[\S\ 11.7]{MR1668103}.
Let $M\subseteq\C^N$ be a CR submanifold with defining functions $\rho =(\rho_1,\dotsc,\rho_d)$ 
near $p_0\in M$.
A \emph{formal holomorphic vector field} at $p_0$ is a vector field of the form
\begin{equation*}
X=\sum_{j=1}^N a_j(Z)\frac{\partial}{\partial Z_j}
\end{equation*}
with the coefficients $a_j$ being formal power series in $Z-p_0$ with complex coefficients.
The formal vector field $X$ is said to be tangent to $M$ at $p_0$ iff there exists a $d\times d$ matrix $c(Z,\bar{Z})$ consisting
of formal power series in the variables $Z-p_0$ and $\bar{Z}-\bar{p}_0$ such that
\begin{equation*}
X\rho(Z,\bar{Z})\sim c(Z,\bar{Z})\rho(Z,\bar{Z}),
\end{equation*}
where $\sim$ denotes equality as formal power series in $Z-p_0$ and $\bar{Z}-\bar{p}_0$.
Note that the existence of nontrivial holomorphic vector fields at $p_0$ tangent to $M$ does not depend on
the choice of holomorphic coordinates and defining equations near $p_0$.
\begin{Def}
	A generic submanifold $M\subseteq\C^N$ is formally holomorphically nondegenerate at $p_0\in M$ iff
	there is no nontrivial formal holomorphic vector field at $p_0$ that is tangent to $M$.
\end{Def}
\begin{Rem}
	If $M$ is formally holomorphically nondegenerate at $p_0$ then $M$ is formally holomorphically nondegenerate 
	at every point of some neighbourhood $U$ of $p_0$.
	Furthermore if $M$ is formally holomorphically nondegenerate on an open set $U\subseteq M$ 
	then $M$ is finitely nondegenerate on an open and dense subset $V\subseteq U$, c.f.\ \cite[Theorem 11.7.5]{MR1668103}.
\end{Rem}
\begin{Thm}
	Let $\M$ be a quasianalytic regular weight sequence and $M\subseteq \C^N$ a generic submanifold of class $\{\M\}$ that is formally holomorphically nondegenerate.
	
	Every smooth CR diffeomorphism $\mathfrak{Y}$ that extends microlocally to a wedge with edge $M$ 
	is ultradifferentiable of class $\{\M\}$.
\end{Thm}
\begin{proof}
	As usual we argue locally near a point $p_0$. 
	After a choice of local bases of CR vector fields and holomorphic forms 
	and selecting a generating set for the characteristic forms we can use the representation \eqref{2LokalRep}
	near $p_0$. 
	By Theorem \ref{thm:multipliers} we know that for any multiplier $\lambda$ the product 
	$\Lambda_j =\lambda\cdot X_j$ is ultradifferentiable for $j=1,\dotsc, N$.
	Since $X_j$ is smooth by assumption we have that the equality holds also for the formal power series at $p_0$
	of $\Lambda_j$, $\lambda$ and $X_j$.
	Since $M$ is formally holomorphically nondegenerate at $p_0$ there has to be a multiplier $\lambda\in\mathcal{S}$
	with nontrivial formal power series at $p_0$. Indeed, if the power series of $\lambda$ at $p_0$ equals $0$ then
	$\lambda$ itself has to vanish in a neighbourhood of $p_0$ by the quasianalyticity of $\M$.
	On the other hand in every neighbourhood of $p_0$ there is a point $q$ at which $M$ is finitely nondegenerate by
	\cite[Theorem 11.7.5]{MR1668103}. 
	Hence by Remark \ref{Mult-FiniteNonDeg} there has to be a nontrivial multiplier $\lambda^\prime$ defined on some
	neighbourhood $U$ of $p_0$. 
	
	We conclude that the formal power series of $\Lambda_j^\prime=\lambda^\prime X_j$ at $p_0$ is divisible by
	the Taylor series of  $\lambda^\prime$ 
	at $p_0$.
	Hence Theorem \ref{FormalDiv} gives that $X_j$ is ultradifferentiable of class $\{\M\}$ near $p_0$.
\end{proof}

\bibliographystyle{plain}
\bibliography{UltraRef1}
\end{document}